\documentclass{article}
\usepackage[utf8]{inputenc}
\usepackage{authblk}
\usepackage[dvipsnames]{xcolor}
\usepackage{fullpage}
\usepackage{amsthm}
\usepackage{amsmath}
\usepackage{amsfonts}
\usepackage{amssymb}
\usepackage{tikz-cd}
\usepackage{graphicx}
\usepackage{subcaption}
\usepackage{dsfont}
\usepackage{todonotes}
\usepackage{tikz-cd}
\usepackage{dsfont}
\usepackage{enumitem}
\usepackage{caption}
\usepackage{subcaption}
\usepackage{thmtools}
\usepackage{thm-restate}
\usepackage{hyperref}
\hypersetup{
    hidelinks,
    colorlinks = true,
    linkcolor = NavyBlue,
    anchorcolor = blue,
    citecolor = PineGreen,
    filecolor = blue,
    urlcolor = Periwinkle
    }
\usepackage[nameinlink]{cleveref}
\usepackage[
backend=bibtex8,
style=alphabetic
]{biblatex}
\newtheoremstyle{OVplain}
{}
{}
{}
{}
{\bfseries}
{.}
{\newline}
{}

\theoremstyle{OVplain}
\newtheorem{thm}{Theorem}[section]

\newtheorem{lem}{Lemma}[section]
\newtheorem{prop}{Proposition}[section]
\newtheorem{cor}{Corollary}[section]
\newtheorem{ex}{Example}[section]

\theoremstyle{OVplain}
\newtheorem{defn}{Definition}[section]

\theoremstyle{remark}
\newtheorem{Remark}{Remark}[section]

\DeclareMathOperator{\Ima}{Im}

\DeclareMathOperator{\id}{id}

\newcommand{\floor}[1]{\lfloor #1 \rfloor}
\newcommand{\ceil}[1]{\lceil #1 \rceil}

\newcommand{\multimodules}{\vec{Vect}^{\vec{R}^n}}

\newcommand{\pfdmultimodules}{\vec{vect}^{\vec{R}^n}}
\newcommand{\fpbimodules}{\vec{vect}_\text{fin}^{\vec{R}^2}}
\newcommand{\Merge}{\textsf{M}_\delta^\mathcal{G}}
\newcommand{\Grid}{\text{Grid}_{\mathcal{G}_M}}
\newcommand{\fpmultimodules}{\vec{vect}_\text{fin}^{\vec{R}^n}}
\newcommand{\Unmerge}{\textsf{U}_\delta^\mathcal{G}}
\newcommand{\unmerge}[2]{\ensuremath{\textsf{U}_{#1}^{#2}}}
\newcommand{\kapeps}{\kappa \varepsilon}

\let\vec\mathbf

\newcommand*{\email}[1]{%
    \normalsize\href{mailto:#1}{#1}\par
    }

\title{Local Equivalence of Metrics for Multiparameter Persistence Modules}
\author{Oliver Vipond}
\affil{\email{vipond@maths.ox.ac.uk} University of Oxford, Andrew Wiles Building, Woodstock Rd, Oxford OX2 6GG}

\begin{document}

\maketitle
\begin{abstract}
    An ideal invariant for multiparameter persistence would be \textit{discriminative}, \textit{computable} and \textit{stable}. In this work we analyse the discriminative power of a stable, computable invariant of multiparameter persistence modules: the fibered bar code. The fibered bar code is equivalent to the rank invariant and encodes the bar codes of the 1-parameter submodules of a multiparameter module. This invariant is well known to be globally incomplete. However in this work we show that the fibered bar code is locally complete for finitely presented modules by showing a local equivalence of metrics between the interleaving distance (which is complete on finitely-presented modules) and the matching distance on fibered bar codes. More precisely, we show that: for a finitely-presented multiparameter module $M$ there is a neighbourhood of $M$, in the interleaving distance $d_I$, for which the matching distance, $d_0$, satisfies the following bi-Lipschitz inequalities $\frac{1}{34}d_I(M,N) \leq d_0(M,N) \leq d_I(M,N)$ for all $N$ in this neighbourhood about $M$. As a consequence no other module in this neighbourhood has the same fibered bar code as $M$.
\end{abstract}

\section{Introduction}

The theory and application of multiparameter persistent homology is a topic of significant interest in the field of Topological Data Analysis. Current work has studied invariants and metrics of multiparameter persistence modules, together with algorithms for their efficient computation \cite{vipond_multiparameter_2018,kerber_exact_2019, dey_computing_2019, lesnick_computing_2019}.

Compared to single parameter persistence, studying data sets filtered over multiple parameters yields a richer class of topological invariants: multiparameter persistence modules. A distinct difference between multiparameter and single parameter persistence modules is that multiparameter modules do not admit a discrete complete invariant analogous to the bar code for single parameter persistence modules \cite{carlsson_theory_2009}. Another upshot of the increased complexity of multiparameter modules, is that the analogous interleaving distance cannot be reduced to a matching distance as in the single parameter case and is thus harder to compute \cite{bjerkevik_computing_2019}.

An ideal invariant for multiparameter persistence would be \textit{discriminative}, \textit{computable} and \textit{stable}.
In earlier work studying invariants for multiparameter persistence, it is common to sacrifice the discriminating power of an invariant in order to achieve computability \cite{vipond_multiparameter_2018,mccleary_multiparameter_2019,corbet_kernel_2019}. Also it is common to study subclasses of persistence modules to tame the wild behaviour of arbitrary multiparameter modules. Such subclasses include: exact \cite{cochoy_decomposition_2020}, interval decomposable \cite{bjerkevik_stability_2016, dey_computing_2019}, block decomposable \cite{botnan_algebraic_2018, cochoy_decomposition_2020, bjerkevik_computing_2019} and rectangular interval decomposable \cite{botnan_rectangle-decomposable_2020} multiparameter persistence modules.
For data science applications, stability is an essential property of a multiparameter module invariant; indeed most data is susceptible to noise and so we desire that our invariants be robust to small perturbations.

A discriminative, stable metric one could choose to study multiparameter persistence modules is the interleaving distance \cite{lesnick_theory_2015}. The theoretical properties of the interleaving distance are well-behaved. Lesnick showed that the interleaving distance is universal on the space of multiparameter modules, and thus the most discriminative stable metric \cite{lesnick_theory_2015}. In recent work, it has been shown that this approach is not feasible computationally. Bjerkevik, Botnan and Kerber showed that the interleaving distance is \textsf{NP}-hard to compute and approximate for multiparameter modules \cite{bjerkevik_computing_2019}. We therefore must find a compromise between discrimating power and computability of multiparameter modules invariants and metrics.

Recently a wealth of computable invariants and metrics have been proposed and studied: \cite{landi_research_2018,cerri_betti_2013,vipond_multiparameter_2018,mccleary_multiparameter_2019,scolamiero_multidimensional_2017,corbet_kernel_2019,harrington_stratifying_2019}. Several of these proposed invariants and metrics have the same discriminating power as the fibered bar code \cite{vipond_multiparameter_2018,mccleary_multiparameter_2019,corbet_kernel_2019}. The fibered bar code encodes all of the single parameter submodules of a multiparameter module and discards higher order interactions. The fibered bar code is well known to be an incomplete invariant, what's more, there exist modules arbitrarily distant in the interleaving distance which have the same fibered bar code. However the fibered bar code and derived invariants are computable, stable and admit a range of desired features for data analysis: $L^p$-norms, averaging, statistical analysis \cite{vipond_multiparameter_2018, the_rivet_developers_rivet_2018, corbet_kernel_2019}.

A significant strength of our work is that the results apply to a constructible subclass of multiparameter modules, those which are finitely presented. This restriction is relevant to data analysis applications, since the sublevel-set persistence module of a finite multifiltered simplicial complex arising from finite data will give rise to such a constructible module. Hence our results apply to the multiparameter modules one is likely to encounter in applications.

\subsubsection*{Our contributions}

This work relates the local behaviour of the theoretically well-behaved interleaving distance $d_I$ to the computable yet incomplete matching distance $d_0$ \cite{landi_research_2018,cerri_betti_2013,kerber_exact_2019}. To the author's best knowledge this is the first place such a comparison has been made. This work provides a positive result that shows a local equivalence between these two metrics.

The main technical result of our paper shows that the matching distance $d_0$ distinguishes a finitely presented ($\mathbb{R}^n$ indexed) multiparameter persistence module $M$ from all modules $N$ in a $d_I$-neighbourhood of $M$ (\Cref{thm:LocalEquivalence}). We show that there is a $d_I$-open ball, $B_M$, centred at $M$ in the space of finitely presented modules such that for all $N \in B_M$:
\begin{equation}\label{eq:LocalEquvalence}
    \frac{1}{34}d_I(M,N) \leq d_0(M,N) \leq d_I(M,N)
\end{equation} 

The radius of the open ball for which \eqref{eq:LocalEquvalence} is valid is dependent on $M$. Moreover the statement is anchored about $M$, the inequalities of \eqref{eq:LocalEquvalence} hold in the ball $B_M$ when the first argument of the distance functions is fixed to be $M$. Indeed there exist $N,N'\in B_M$ such that $\frac{1}{34}d_I(N',N) \not\leq d_0(N',N)$. 

The upper bound $d_0 \leq d_I$ is not new and is a constructed property of the matching distance $d_0$. The local lower bound $\frac{1}{34}d_I\leq d_0$ is proven in this work and gives rise to a global equivalence of intrinsic metrics \Cref{thm:GlobalEquivalence}.

We observe that the space of $\mathbb{R}^n$-indexed finitely presented multiparameter persistence modules is a path metric space when equipped with the interleaving distance. This property follows from the characterisation of interleavings proven by Lesnick \cite{lesnick_theory_2015}. However, to the best of the author's knowledge, the fact that $d_I$ is an intrinsic metric has not been explored nor formally stated in the multiparameter persistence literature to date.

The local result \eqref{eq:LocalEquvalence} extends to a global result for the induced path metric $\hat{d}_0$ (\Cref{defn:IntrinsicMetric}). This says that $\hat{d}_0$ and $d_I$ are bi-Lipschitz equivalent as metrics on the whole space of finitely presented multiparameter persistence modules (\Cref{thm:GlobalEquivalence}). For any pair of finitely presented multiparameter modules $M,N$ we attain:

\begin{equation}
    \frac{1}{34}d_I(M,N) \leq \hat{d}_0(M,N) \leq d_I(M,N) 
    \label{eq:GlobalEquvalence}
\end{equation} 

Whilst $d_I$ is an intrinsic metric (\Cref{cor:PathMetricSpace}), the collection of geodesics between any pair of distinct modules is highly non-unique. This is a consequence of the interleaving distance behaving as an $L^\infty$-style norm. The matching distance $d_0$ also behaves like an $L^\infty$-style norm, however there are $L^p$-style norms defined on faithful invariants derived from the fibered bar code \cite{vipond_multiparameter_2018,corbet_kernel_2019}. These $L^p$-style metrics constrain the collection of geodesics between modules. This local equivalence result is the first step towards a sensible notion of interpolation between multiparameter modules compatible with the interleaving distance.

Finally, we identify that our result for multiparameter modules attains as a simple corollary equivalences of metrics on constructible spaces which can be embedded into the space of finitely presented modules. In particular we show how interlevel set persistence modules give rise to finitely presented 2-parameter persistence modules. 

\subsubsection*{Related work}

The author has been interested in comparing computable metrics on multiparameter persistence modules with the interleaving distance. However, a major inspiration for this work was the article of Carriere and Oudot: \textit{Local Equivalence and Intrinsic Metrics Between Reeb Graphs} \cite{carriere_local_2017}. In this work Carriere and Oudot conjecture that similar results to their local equivalence of metrics on Reeb graphs apply to ``more general classes of metric spaces than Reeb graphs". Indeed in \Cref{sec:InterlevelReeb} we show that our local equivalence result for multiparameter modules induces a similar but weaker result to the main result of \cite{carriere_local_2017} for Reeb graphs (induced by embedding Reeb graphs into multiparameter persistence modules). In our proof of \Cref{thm:LocalEquivalence} we employ similar techniques to those used in \cite{carriere_local_2017}.

The work of Lesnick and Wright has been fundamental to the author's understanding of finitely presented multiparameter persistence modules and the interleaving distance \cite{lesnick_theory_2015,lesnick_computing_2019, lesnick_interactive_2015}. In particular we have used in this work that a free resolution of a multiparameter module induces a free resolution of the 1-parameter submodules, the characterisation of interleavings (\Cref{thm:Characterisation}), and the push function of \cite{lesnick_interactive_2015}.

In extending our local equivalence result \Cref{thm:LocalEquivalence} to a global equivalence \Cref{thm:GlobalEquivalence} we consider paths in the space of multiparameter persistence modules in order to induce intrinsic metrics. The resulting intrinsic metric is similar in nature to the Wasserstein distance in \cite{bubenik_wasserstein_2018}, which is defined via a path of simple morphisms.

\subsection*{Acknowledgements}
The author would like to thank his supervisors Ulrike Tillmann and Vidit Nanda for their support during his research. In addition, the author would like to thank Jacob Leygonie for many constructive discussions which cemented the technical arguments in this work. The author is grateful to receive support from EPSRC studentship EP/N509711/1 and EPSRC grant EP/R018472/1.  

\subsection*{Outline of Content}

In \Cref{sec:Prelims} we introduce multiparameter persistence theory including the two metrics which we wish to compare: the interleaving distance (\Cref{def:InterleavingDistance}) and the matching distance (\Cref{def:MatchingDistance}).

In \Cref{sec:M&SFunctors} we define Merge and Simplification Functors which we shall use to manipulate finitely presented multiparameter persistence modules. We prove these functors are well defined and establish their effect on the presentation of a finitely presented multiparameter persistence module.

\Cref{sec:LocalEquivalence} is the technical heart of the paper within which we prove the local equivalence of the interleaving distance and matching distance (\Cref{thm:LocalEquivalence}). The proof of this result makes heavy use of the functors defined and established in \Cref{sec:M&SFunctors}.

In \Cref{sec:GlobalEquiv} we recall the definition of intrinsic metrics. We show that the space of finitely presented multiparameter persistence modules has geodesic paths with respect to the interleaving distance, and thus our local equivalence result induces a global equivalence between the interleaving distance and the intrinsic matching distance.

Finally in \Cref{sec:InterlevelReeb} we show that our local equivalence of metrics result for multiparameter persistence modules induces a local equivalence of metrics for Reeb graphs analogous to the result attained in \cite{carriere_local_2017}.

\section{Preliminaries}

\label{sec:Prelims}

In this section we tersely introduce the key multiparameter persistence definitions we require to develop the functors in \Cref{sec:M&SFunctors} and prove the main results in \Cref{sec:LocalEquivalence} and \Cref{sec:GlobalEquiv}. For a more complete introduction see \cite{carlsson_theory_2009,lesnick_interactive_2015}. 

\subsection{Multiparameter Persistence Modules}

\subsubsection*{Notation}

Let $P_n$ denote the monoid ring of the monoid $([0,\infty)^n,+)$ over a field $\mathbb{F}$. One can think of $P_n$ as a pseudo-polynomial ring $\mathbb{F}[x_1,...,x_n]$ in which exponents are only required to be non-negative and can be non-integral. We shall denote the monomial $\Pi_{i=1}^n x_i^{a_i}\in P_n$ as $\vec{x}^\vec{a}$.
Let $\vec{R}^n$ denotes the category associated to the poset $(\mathbb{R}^n,\leq)$ under the standard coordinate-wise partial order, and more generally let  $\vec{P}$ denote the category associated to the poset $P$. Let $\vec{Vect}$ denote the category of vector spaces and linear maps over $\mathbb{F}$, and $\vec{vect}$ denote the subcategory of finite dimensional vector spaces. For a category $\vec{C}$ and a poset category $\vec{P}$ let us denote the functor category of $\textbf{C}$-valued functors on $\textbf{P}$ by $\textbf{C}^\textbf{P}$. For a poset $P$ we shall use $\vee$ to denote the join of elements.

\begin{defn}[Multiparameter Persistence Module]
A \textcolor{Maroon}{multiparameter persistence module} is an $\mathbb{R}^n$-graded $P_n$-module, normally denoted by $M$. That is to say $M$ has a decomposition as a $\mathbb{F}$-vector space $M = \bigoplus_{\vec{a}\in \mathbb{R}^n} M_\vec{a}$ compatible with the action of $P_n$: if $\vec{a}\in \mathbb{R}^n$, $\vec{x}^\vec{b}\in P_n$ and $m \in M_\vec{a}$ then $\vec{x}^\vec{b} \cdot m \in M_{\vec{a}+\vec{b}}$.
A \textcolor{Maroon}{morphism} of graded modules is required to respect the grading and be compatible with the module structure i.e. if $f:M \to N$, $r\in P_n$ and $m\in M_\vec{a}$ then $f(m)\in N_\vec{a}$ and $r\cdot f(m) = f(r\cdot m)$.
\end{defn}

\noindent For the statement of some results it is more convenient to use the following equivalent category-theoretic definition.

\begin{defn}[Multiparameter Persistence Module]
A \textcolor{Maroon}{multiparameter persistence module} is an element of the functor category $\vec{Vect}^{\vec{R}^n}$. A \textcolor{Maroon}{morphism} of multiparameter persistence modules is a natural transformation $M \Rightarrow M'$.
\end{defn}

The equivalence of the two perspectives is simply saying that we have an equivalence of categories between $\mathbb{R}^n$-graded $P_n$-$\vec{Mod}$ and $\vec{Vect}^{\vec{R}^n}$ \cite{lesnick_theory_2015}. We shall freely switch between these equivalent perspectives throughout.

\begin{defn}[Internal Translation]
We shall denote by $\varphi^M_\varepsilon$ the \textcolor{Maroon}{internal translation} of the module $M$ given by the action of the ring element $\vec{x}^{\varepsilon\vec{1}}\in P_n$ on $M$: $\varphi^M_\varepsilon: M \to M$ via $\varphi^M_\varepsilon(m) = \vec{x}^{\varepsilon\vec{1}}\cdot m$.
\end{defn}

\begin{defn}[Interval Decomposable Modules]
\label{defn:intervalmodule}
Let $I \subset \mathbb{R}^n$ be a connected subposet such that $\vec{a},\vec{b}\in I$ and $\vec{a}\leq \vec{p} \leq \vec{b}$ implies $\vec{p} \in I$, then we say $I$ is an \textcolor{Maroon}{interval}. Let $\mathds{1}^I$ denote the $\vec{vect}$-valued functor with domain $\vec{R}^n$ for which $\dim \mathds{1}^I(\vec{a}) = \mathds{1}\{\vec{a} \in I \}$ and such that the morphisms $\mathds{1}^I(\vec{a}\leq \vec{b})$ are isomorphisms wherever possible. We say that $\mathds{1}^I$ is an \textcolor{Maroon}{interval module}, and any module $M$ which is isomorphic to the direct sum of interval modules is \textcolor{Maroon}{interval decomposable}.
\end{defn}

\noindent We shall regularly use interval decomposable modules in our examples and figures. However one should keep in mind that our results are not constrained to this subclass of modules.

Let us now define the interleaving distance, the most discriminative stable metric one can consider on multiparameter persistence modules alluded to in the introduction. We shall adopt the notation of \cite{bubenik_metrics_2015}. Let $T_\varepsilon : \vec{R}^n \to  \vec{R}^n$ denote the translation endofunctor given by $T_\varepsilon(\vec{a}) = \vec{a} + \varepsilon\vec{1}$ where $\vec{1} = (1,1,...,1)$. Observe that this translation induces an endofunctor $T^\ast_{\varepsilon}: \multimodules \to \multimodules$ where $T^\ast_{\varepsilon}(M) =  M\circ T_{\varepsilon}$.

\begin{defn}[Interleaving Distance]
\label{def:InterleavingDistance}
Let $M,N \in \multimodules$ be multiparameter persistence modules and  We say that $M,N$ are \textcolor{Maroon}{$\varepsilon$-interleaved} if there exist natural transformations $f : M \Rightarrow NT_\varepsilon$, $g: N\Rightarrow MT_\varepsilon$ satisfying the coherence criteria that $T^\ast_{\varepsilon}(g)f  = \varphi_{2\varepsilon}^M $, $ T^\ast_{\varepsilon(f})g  = \varphi_{2\varepsilon}^N $.

We define the \textcolor{Maroon}{interleaving distance} to be the infimum of $\varepsilon$ for which $M,N$ are $\varepsilon$-interleaved:

$$ d_I(M,N) = \inf \{\varepsilon\geq 0 : M,N \text{ are } \varepsilon\text{-interleaved}\}$$
\end{defn}

An interleaving may be thought of as an approximate isomorphism. Modules are $0$-interleaved if and only if they are isomorphic. By \textit{blurring} the poset with translations $T_{\varepsilon}$ we admit flexibility to the rigid notion of isomorphism.
The interleaving distance is universal amongst stable distances on persistence modules, that is to say any other stable distance is bounded above by the interleaving distance \cite{lesnick_theory_2015}. However, the interleaving distance has been shown to be \textsf{NP}-hard to compute and approximate \cite{bjerkevik_computing_2019}.

Associated to a multiparameter module is a family of single parameter modules whose collection of bar codes is known as the fibered bar code. A line in $\mathbb{R}^n$ with equation $\mathcal{L}(t) = t\vec{m}+\vec{c}$ is said to be positively sloped if each coordinate of the gradient $\vec{m}$ is strictly positive ($m_i>0$ for all $i$). 

\begin{defn}[Fibered Bar Code]

Let $\vec{{L}}$ denote the subposet of $\vec{R}^n$ corresponding to a positively sloped line $\mathcal{L}\subset \mathbb{R}^n$. Let $\iota_\mathcal{L} : (\mathbb{R},\|\cdot\|_\infty) \to (\mathbb{R}^n,\|\cdot\|_\infty)$ denote the isometric embedding with $\iota_\mathcal{L}( \vec{R})=\vec{L}$ and $\iota_\mathcal{L}(0)\in \{\vec{x} \in \mathbb{R}^n : x_n = 0\}$. Then for $M\in\vec{Vect}^{\vec{R}^n}$ the composite $M^\mathcal{L} = M\circ \iota_\mathcal{L}$ is a single parameter persistence module, and thus has an associated bar code $\mathcal{B}(M^\mathcal{L})$. The \textcolor{Maroon}{fibered bar code} of $M$ is the collection $\{\mathcal{B}(M^\mathcal{L}) \ :\  \mathcal{L} \in \Lambda \}$ where $\Lambda$ denotes the set of positively sloped lines.
\end{defn}

In contrast to the interleaving distance, the fibered bar code is computable for a (finitely-presented) multiparameter persistence module. The RIVET software \cite{the_rivet_developers_rivet_2018} efficiently computes the fibered bar code for a 2-parameter finitely presented module, and has been used in \cite{keller_phos:_2018,vipond_multiparameter_2018}.

\begin{figure}
    \centering
    \begin{subfigure}{0.45\textwidth}
        \centering
        \includegraphics[width=1\textwidth]{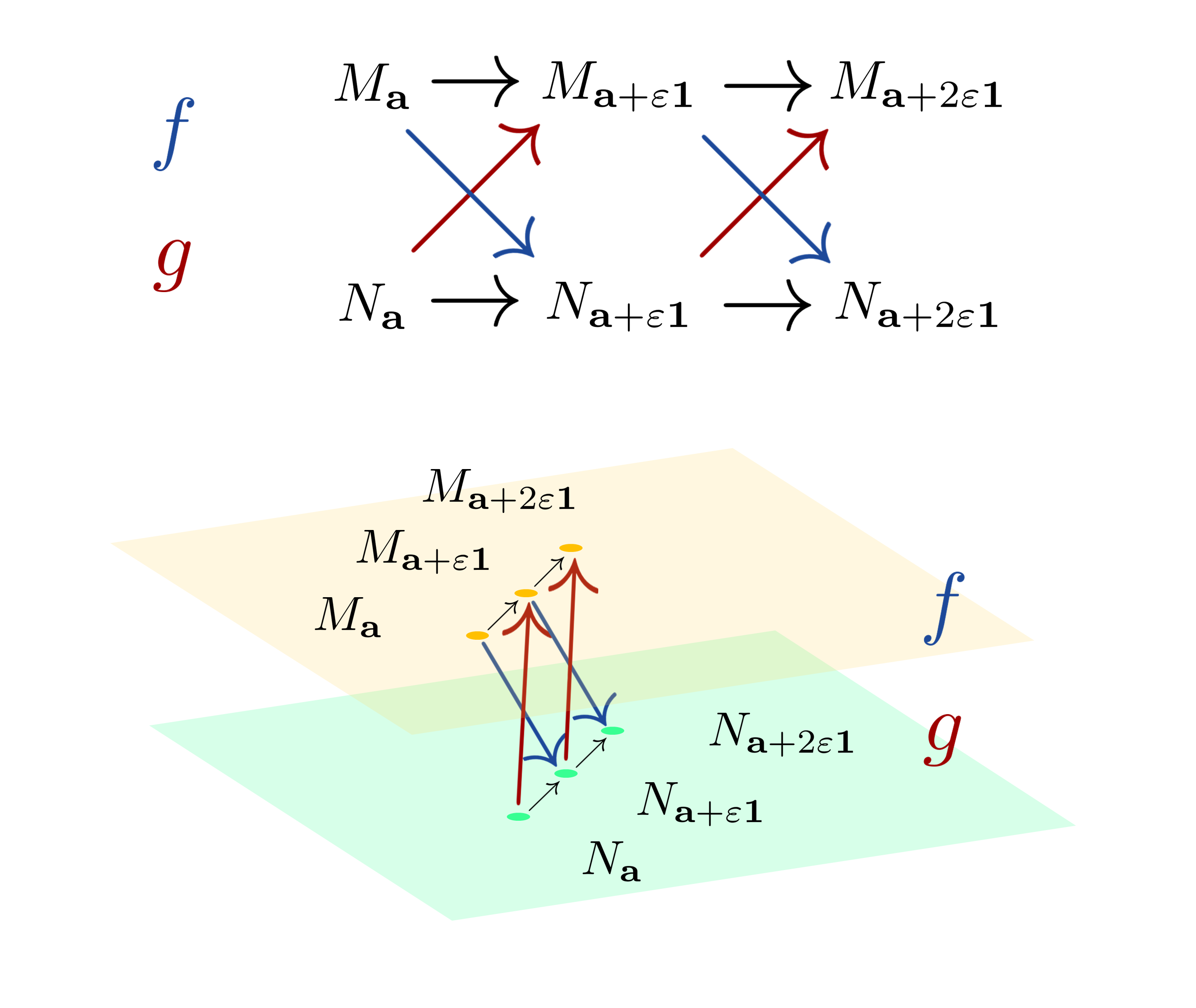} 
        \caption{A visualisation of the morphisms comprising an $\varepsilon$-interleaving for 2-parameter persistence modules.}
        \label{subfig:FiberedBarcode}
    \end{subfigure}\hfill
    \begin{subfigure}{0.45\textwidth}
        \centering
        \includegraphics[width=1\textwidth]{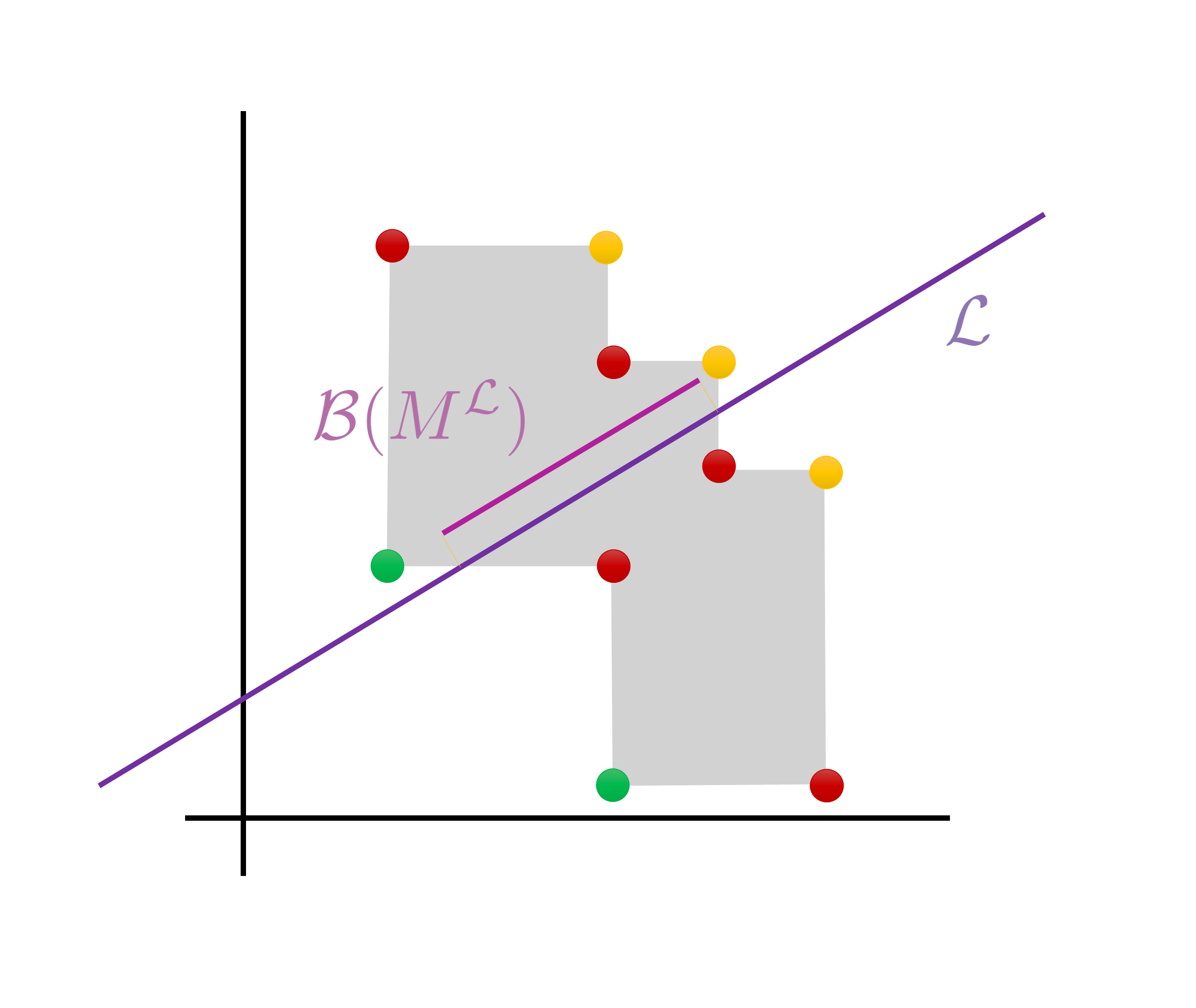} 
        \caption{A bar code $\mathcal{B}(M^\mathcal{L})$ in the fibered bar code of a 2-parameter interval module $M$.}
    \label{subfig:Interleaving}
    \end{subfigure}
    \caption{An illustration of an $\varepsilon$-interleaving and the fibered bar code, for 2-parameter persistence modules.}
    \label{fig:Interleaving&FiberedBarcode}
\end{figure}

\begin{defn}[Push Function \cite{lesnick_interactive_2015}]
\label{defn:pushfunction}
Let $\mathcal{L}: \mathbb{R} \to \mathbb{R}^n$ be a strictly positively sloped line in $\mathbb{R}^n$ isometrically embedded with respect to $\|\cdot\|_\infty$. The map \textcolor{Maroon}{$\text{push}_\mathcal{L}: \mathbb{R}^n \to  \Ima \mathcal{L} $} is defined as $\text{push}_\mathcal{L}(\vec{p}) = \min \{\vec{a} \in \Ima\mathcal{L} : \vec{a} \geq \vec{p}\}$, and is a partial order preserving map, pushing every element of $\mathbb{R}^n$ to the line $\mathcal{L}$.

\end{defn}

Using the bottleneck distance for 1-parameter persistence modules on each 1-parameter submodule we induce a matching distance on the fibered bar code. The matching distance was defined for multiparameter persistence modules in \cite{cerri_betti_2013}, and the stability of the matching distance with respect to the interleaving distance explicitly shown in \cite{landi_research_2018}.

\begin{defn}[Matching Distance  \cite{cerri_betti_2013}]\label{def:MatchingDistance}

Let $M,N\in \vec{vect}^{\vec{R}^2}$ the \textcolor{Maroon}{matching distance} is taken to be the weighted supremum of the bottleneck distance over $1$-dimensional submodules:
$$ d_0(M,N) = \sup_{\mathcal{L}\in \Lambda} \ w(\mathcal{L})\cdot d_B(M^\mathcal{L},N^\mathcal{L})$$
where the weighting $ w(\mathcal{L}) := \| \text{push}_\mathcal{L}(\mathcal{L}(0)+ \vec{1}) - \mathcal{L}(0) \|_\infty$

\end{defn}

The weighting for the matching distance is constructed in order that the matching distance is bounded by the interleaving distance, $d_0\leq d_I$. An $\varepsilon$-interleaving of multiparameter persistence modules $M,N$ induces a $\frac{\varepsilon}{w(\mathcal{L})}$-interleaving between the single parameter modules along the line $\mathcal{L}$, (hence a $\frac{\varepsilon}{w(\mathcal{L})}$-matching between their bar codes), and thus $d_B(M^\mathcal{L},N^\mathcal{L}) \leq d_I(M,N)$ \cite{landi_research_2018,cerri_betti_2013}(see \Cref{fig:InducedInterleaving}).
For example, a line $\mathcal{L}$ of slope $\vec{1}$ has weighting $ w(\mathcal{L})=1$, and the weighting for the line $\mathcal{L} : y = mx+c \subset \mathbb{R}^2$ is given by:
\begin{align*}
    w(\mathcal{L}) = \begin{cases}
    \frac{1}{\sqrt{1+m^2}} & \text{ for } m \geq 1 \\
    \frac{1}{\sqrt{1+\frac{1}{m^2}}} & \text{ for } m < 1 
    \frac{}{}
    \end{cases}
\end{align*}
Exact computation of the matching distance for bimodules has been studied in \cite{kerber_exact_2019}.
\begin{figure}
\centering

\includegraphics[width=0.4\linewidth]{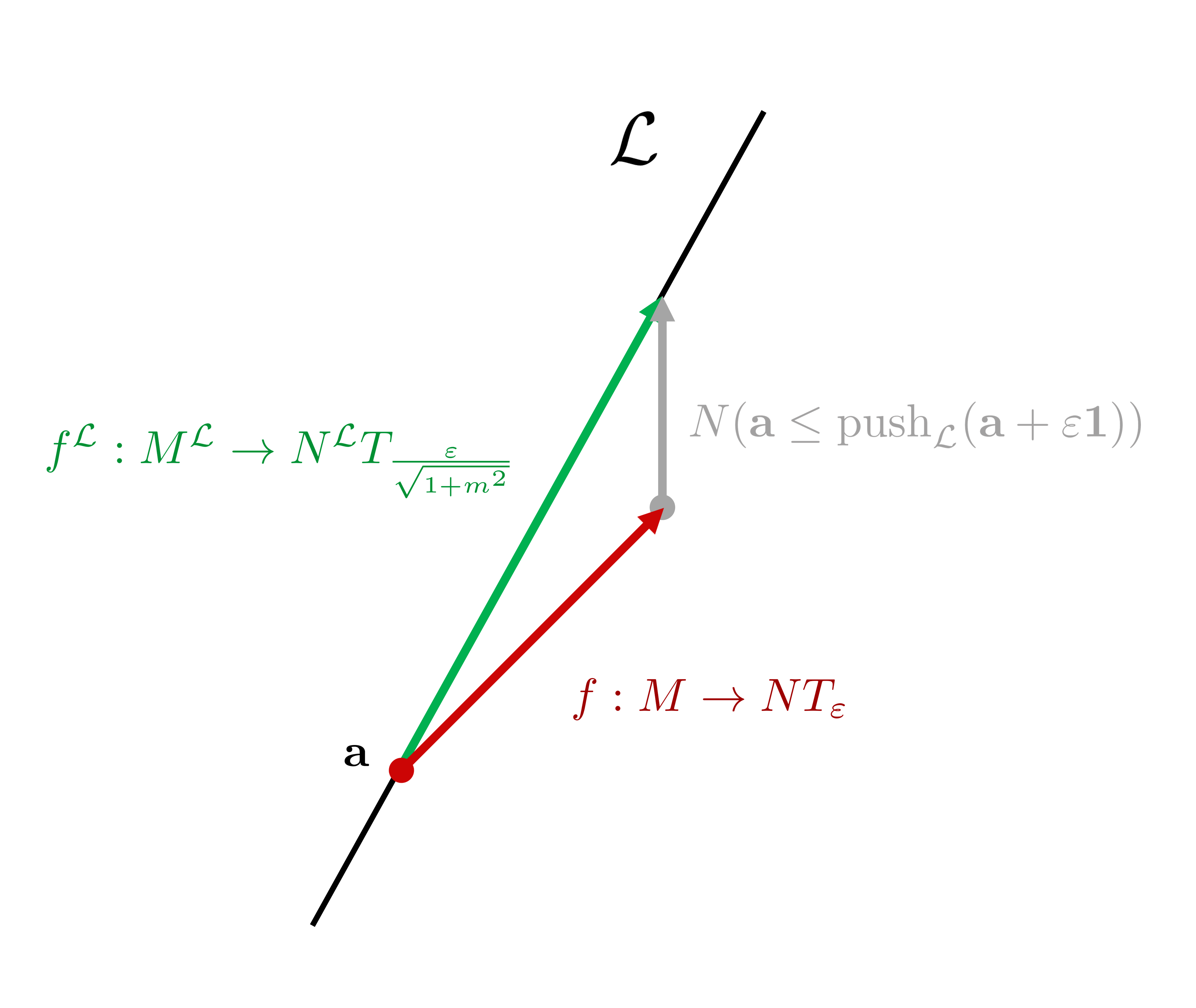}

\caption{An $\varepsilon$-interleaving morphism $f:M \to NT_\varepsilon$ for 2-parameter persistence modules $M$ and $N$ induces a $\frac{\varepsilon}{\sqrt{1+m^2}}$-interleaving morphism $f^\mathcal{L}:M^\mathcal{L} \to N^\mathcal{L}T_\frac{\varepsilon}{\sqrt{1+m^2}}$ for the line $\mathcal{L}: x_2 = m\cdot x_1 +c$ with slope $m\geq 1$. The interleaving morphism $f^\mathcal{L}$ is realised as a composition of the interleaving morphism $f$ and an internal morphism of $N$: $f^\mathcal{L}_\vec{a}= N(\vec{a}+\vec{\varepsilon}\leq \text{push}_\mathcal{L}(\vec{a}+\vec{\varepsilon}))\circ f_\vec{a}$. }
\label{fig:InducedInterleaving}
\end{figure}

\subsection{Presentations}

Let us now develop the theory required to define presentations of multiparameter persistence modules. 




We define an $\mathbb{R}^n$-graded set to be an indexing set $\mathcal{X}$ together with a grading map $\text{gr}: \mathcal{X} \to \mathbb{R}^n$. For an element $j$ of a graded set with $\text{gr}(j)=\vec{a}$, we shall refer to $\vec{a}$ as the grade of $j$. For a graded set $(\mathcal{X},\text{gr})$ we will use $\mathcal{X}(\varepsilon)$ to denote the graded set $(\mathcal{X},T_{-\varepsilon}\circ\text{gr})$. Similarly for $j\in \mathbb{R}^n$ we will use $P_n(-j)$ to denote the the ring $P_n$ with grading shifted by $j$ so that the multiplicative identity of $P_n(-j)$ lives at grade $j$.

\begin{defn}[Free Module]

The \textcolor{Maroon}{free module} on $\mathcal{X}$, an $\mathbb{R}^n$-graded set, is defined to be:
$$ \text{Free}[\mathcal{X}] = \bigoplus_{j\in \mathcal{X}} P_n(-\text{gr}(j))$$
\end{defn}

A free module on a graded set can equivalently be defined using a universal property characterisation \cite{carlsson_theory_2009}, and it is clear that a morphism from a free module is determined by its image on the generating set.

We say a subset $\mathcal{R} \subset M$ of a persistence module is homogeneous if $\mathcal{R} \subset \cup_{\vec{a}\in \mathbb{R}^n} M_\vec{a}$. That is to say each element has a well-defined grade.

\begin{defn}[Presentations]

Let $\mathcal{X}$ be a graded set and $\mathcal{R}$ a homogeneous subset of the free module on $\mathcal{X}$ generating the submodule $\langle \mathcal{R} \rangle$. We say that a persistence module $M$ has \textcolor{Maroon}{presentation} $\langle \mathcal{X} | \mathcal{R}\rangle$ if:
$$ M \cong \frac{\text{Free}[\mathcal{X}]}{\langle \mathcal{R} \rangle}$$

\end{defn}

A presentation is \textcolor{Maroon}{finite} if both $\mathcal{X}$ and $\mathcal{R}$ are finite. We shall denote the subspace of finitely presented mulitparameter modules as \textcolor{Maroon}{$\text{vect}_\text{fin}^{\vec{R}^n}$}.
Let $I$ denote the ideal of $P_n$ generated by the elements $\{\vec{x}^\vec{a} \ | \ \vec{a} > 0 \}$ and let $\Phi_{\langle \mathcal{X} | \mathcal{R}\rangle} : \text{Free}[\mathcal{R}] \to \text{Free}[\mathcal{X}]$ be the map induced by the inclusion $\mathcal{R} \hookrightarrow \text{Free}[\mathcal{X}]$. We say that a presentation of $M$ is  \textcolor{Maroon}{minimal} if $\mathcal{R} \subset I \cdot \text{Free}[\mathcal{X}]$ and $\ker_{\Phi_{\langle \mathcal{X} | \mathcal{R}\rangle}} \subset I \cdot \text{Free}[\mathcal{R}]$.

More generally we say that a free resolution $F_\bullet \to_{p_\bullet} M$ is minimal if $F_i = \text{Free}[\mathcal{X}_i]$ and $\ker(p_i) \subset I \cdot \text{Free}[\mathcal{X}_i]$ for all $i\in \mathbb{N}$.

\begin{figure}
    \centering
        \includegraphics[width=0.5\textwidth]{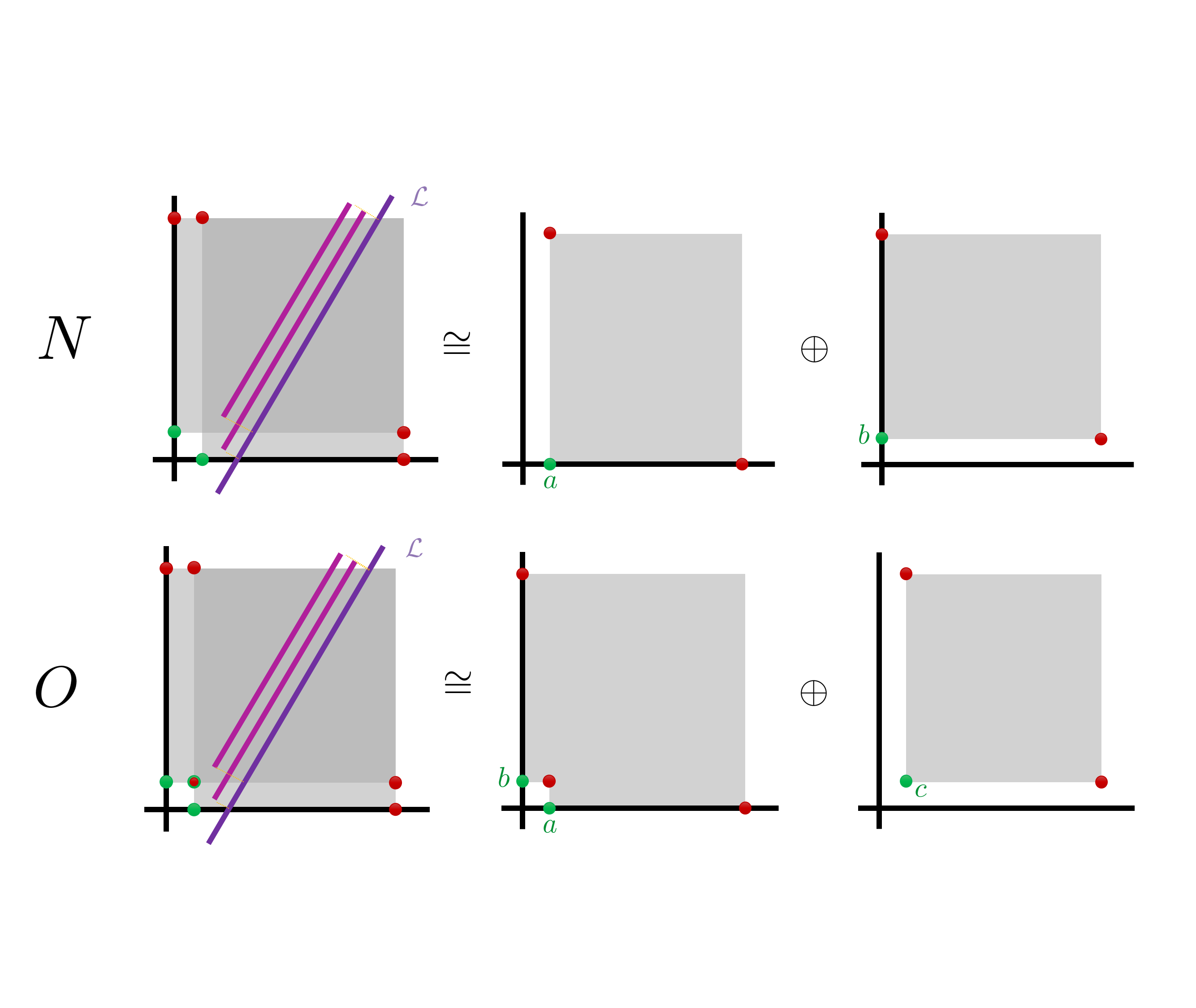}
    \caption{Non-isomorphic interval decomposable 2-parameter persistence modules $N$ and $O$, $d_I(N,O) = \varepsilon > 0$ \Cref{ex:Incompleteness}. The fibered bar codes of $N$ and $O$ are identical and thus the matching distance between $N$ and $O$ is zero, $d_0(N,O)=0$.}
    \label{fig:IncompletenessExample}
\end{figure}

\begin{ex}[Incompleteness of the Matching Distance]
\label{ex:Incompleteness}
Let us define the following graded sets and homogeneous subset:
\begin{align*}
     \mathcal{X}_N &= \{(a,(\varepsilon,0)),(b,(0,\varepsilon))\} \\
     \mathcal{X}_O &= \mathcal{X}_N \cup \{(c,(\varepsilon,\varepsilon))\} \\
     \mathcal{R}_N &= \{x_1^{9\varepsilon}\cdot a,x_2^{9\varepsilon}\cdot b,x_2^{10\varepsilon}\cdot a, x_1^{10\varepsilon}\cdot b\} \subset \text{Free}[\mathcal{X}_N]\\
     \mathcal{R}_O &= \mathcal{R}_O \cup \{x_1^\varepsilon\cdot b - x_2^\varepsilon \cdot a, x_1^{9\varepsilon}\cdot c,x_2^{9\varepsilon}\cdot c \} \subset \text{Free}[\mathcal{X}_O]
\end{align*}

Let $N,O$ denote the 2-parameter persistence modules with presentations $\langle \mathcal{X}_N | \mathcal{R}_N\rangle, \langle \mathcal{X}_O | \mathcal{R}_O\rangle  $ respectively (see \Cref{fig:IncompletenessExample}). The fibered bar code of $N$ and $O$ are indistinguishable, $d_0(N,O) = 0$, yet $d_I(N,O) = \varepsilon$. We also have that $d_I(N,0),d_I(O,0) \leq 5 \varepsilon$, so for any 2-parameter module $M$, $d_I(M,M\oplus N),d_I(M,M\oplus O) \leq 5 \varepsilon$ and $d_0(M\oplus N, M \oplus O) = 0$. This example witnesses the incompleteness of the matching distance and shows that our local equivalence of metrics between $d_I$ and $d_0$ \eqref{eq:LocalEquvalence} is necessarily anchored about $M$. 
\end{ex}

\begin{defn}[Multiparameter Betti Numbers]

Let $M$ be a multiparameter persistence module. The \textcolor{Maroon}{multiparameter Betti numbers} are maps $\xi_i(M) : \mathbb{R}^n \to \mathbb{N}$ defined by:
$$ \xi_i(M)(\vec{a}) = \dim_{\mathbb{F}} (\text{Tor}^{P_n}_i(M,P_n / I P_n)_\vec{a})$$
The multiparameter Betti numbers are well defined (see \cite{lesnick_interactive_2015} for details).

\end{defn}


If $\langle \mathcal{X} | \mathcal{R}\rangle$ is a minimal presentation for $M$ then $\xi_0(M)(\vec{a}) = |\text{gr}^{-1}_\mathcal{X}(\vec{a})|$ and $\xi_1(M)(\vec{a}) = |\text{gr}^{-1}_\mathcal{R}(\vec{a})|$. More generally the support of the maps  $\xi_i(M)$ are the locations of the generators of $F_i$ in a minimal free resolution $F_\bullet \to M$ (up to multiplicity). 

For convenience we will refer to $\xi_i(M)$ as a set, by which we mean the support of the map $\xi_i(M)$. That is, when we refer to a multiparameter Betti number $\vec{a}\in \xi_i(M)$ we mean that $\vec{a}$ is in the support of $\xi_i(M)$ and thus there is a generator at grade $\vec{a}$ in any minimal free resolution.

For a finitely presented module we can break up $\mathbb{R}^n$ into cells within which the module is constant (the internal morphisms are isomorphisms). The boundaries of these cells are determined by the positions of the multiparameter Betti numbers.

\begin{figure}
    \centering
    \begin{subfigure}{0.45\textwidth}
        \centering
        \includegraphics[width=1\textwidth]{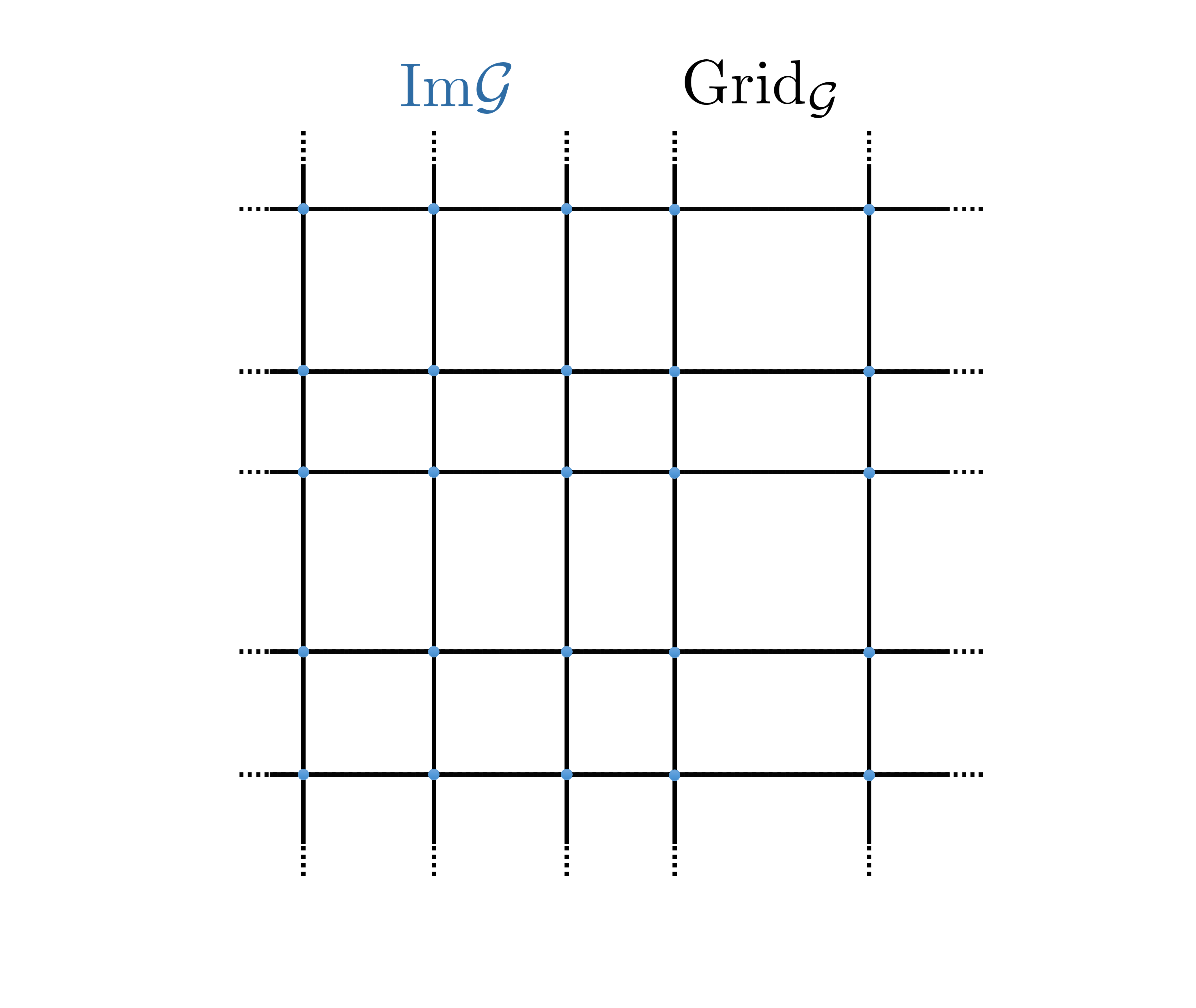} 
        \caption{A grid functions $\mathcal{G}: [k_1]\times [k_2] \to \mathbb{R}^2$ together with the associated collection of axis aligned lines $\text{Grid}_\mathcal{G}$.}
    \end{subfigure}\hfill
    \begin{subfigure}{0.45\textwidth}
        \centering
        \includegraphics[width=1\textwidth]{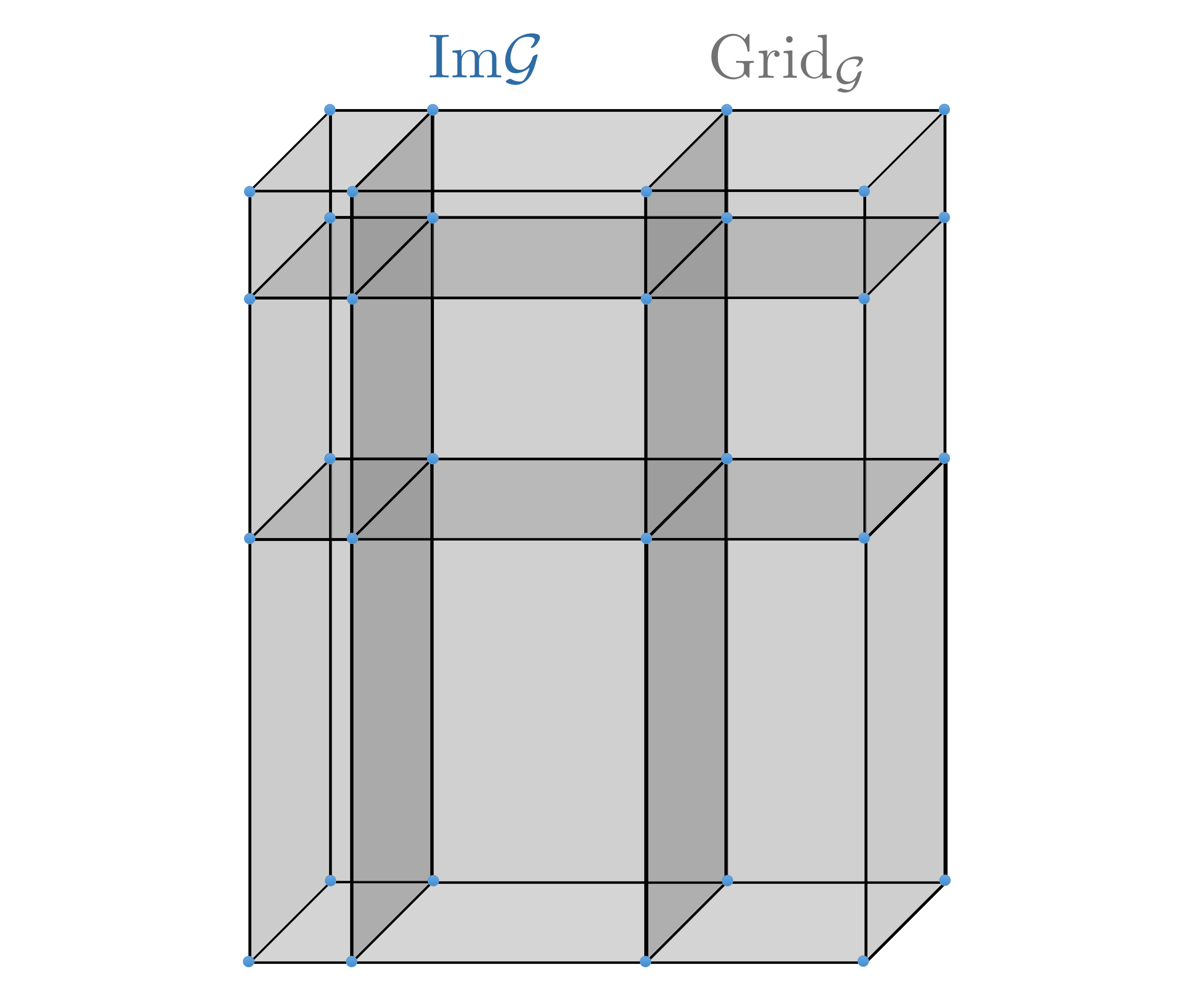} 
        \caption{A grid functions $\mathcal{G}: [k_1]\times [k_2] \times [k_3] \to \mathbb{R}^3$ together with the associated collection of axis aligned hyperplanes $\text{Grid}_\mathcal{G}$.}
    \end{subfigure}
    \caption{Example grid functions with their associated collection of axis aligned hyperplanes.}
    \label{fig:Grids}
\end{figure}

\begin{figure}
    \centering
    \begin{subfigure}{0.45\textwidth}
        \centering
        \includegraphics[width=1\textwidth]{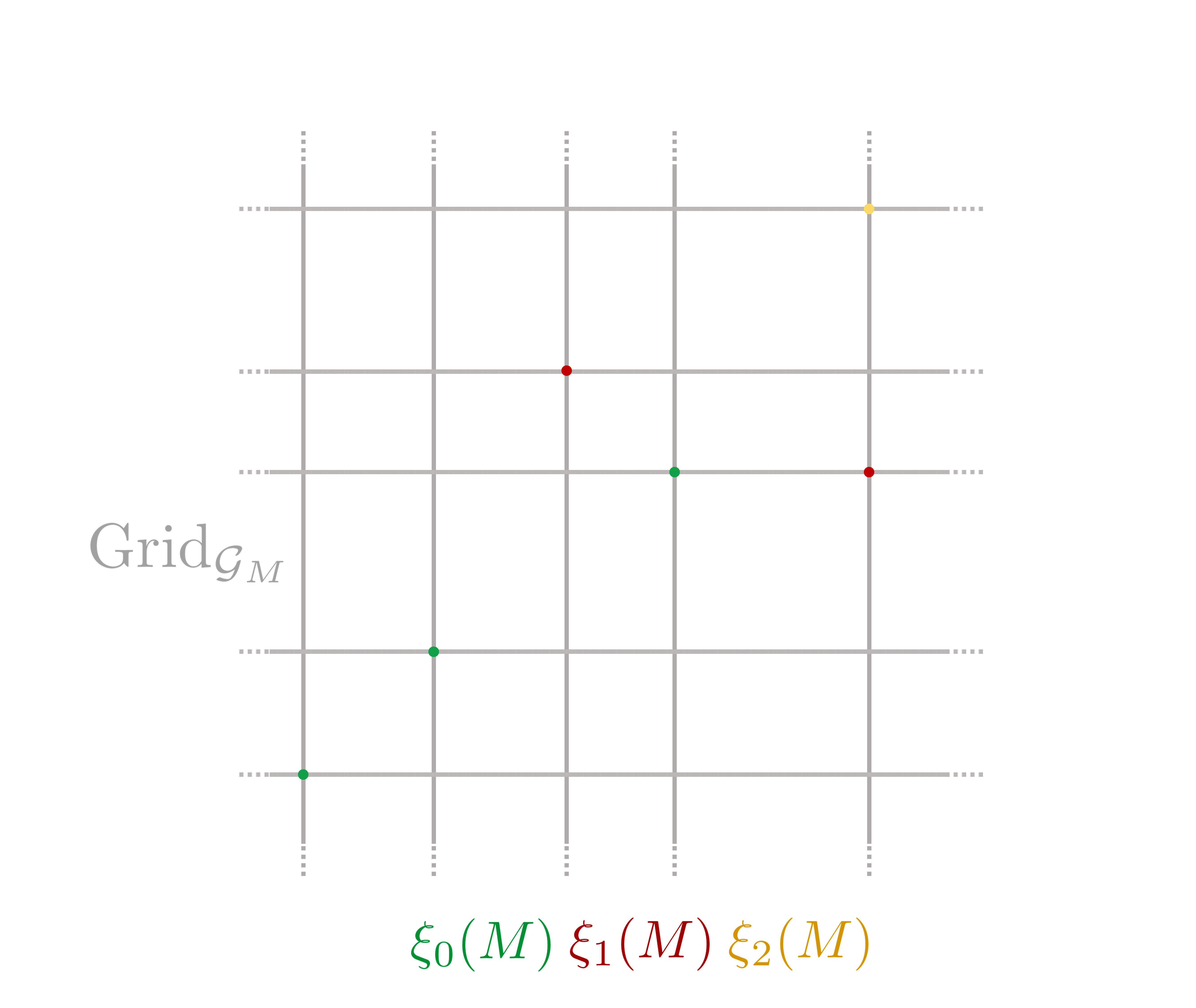} 
        \caption{Multiparameter Betti numbers of some 2-parameter module $M\in \fpbimodules$ together with the collection of axis aligned lines $\text{Grid}_{\mathcal{G}_M}$.}
    \end{subfigure}\hfill
    \begin{subfigure}{0.45\textwidth}
        \centering
        \includegraphics[width=1\textwidth]{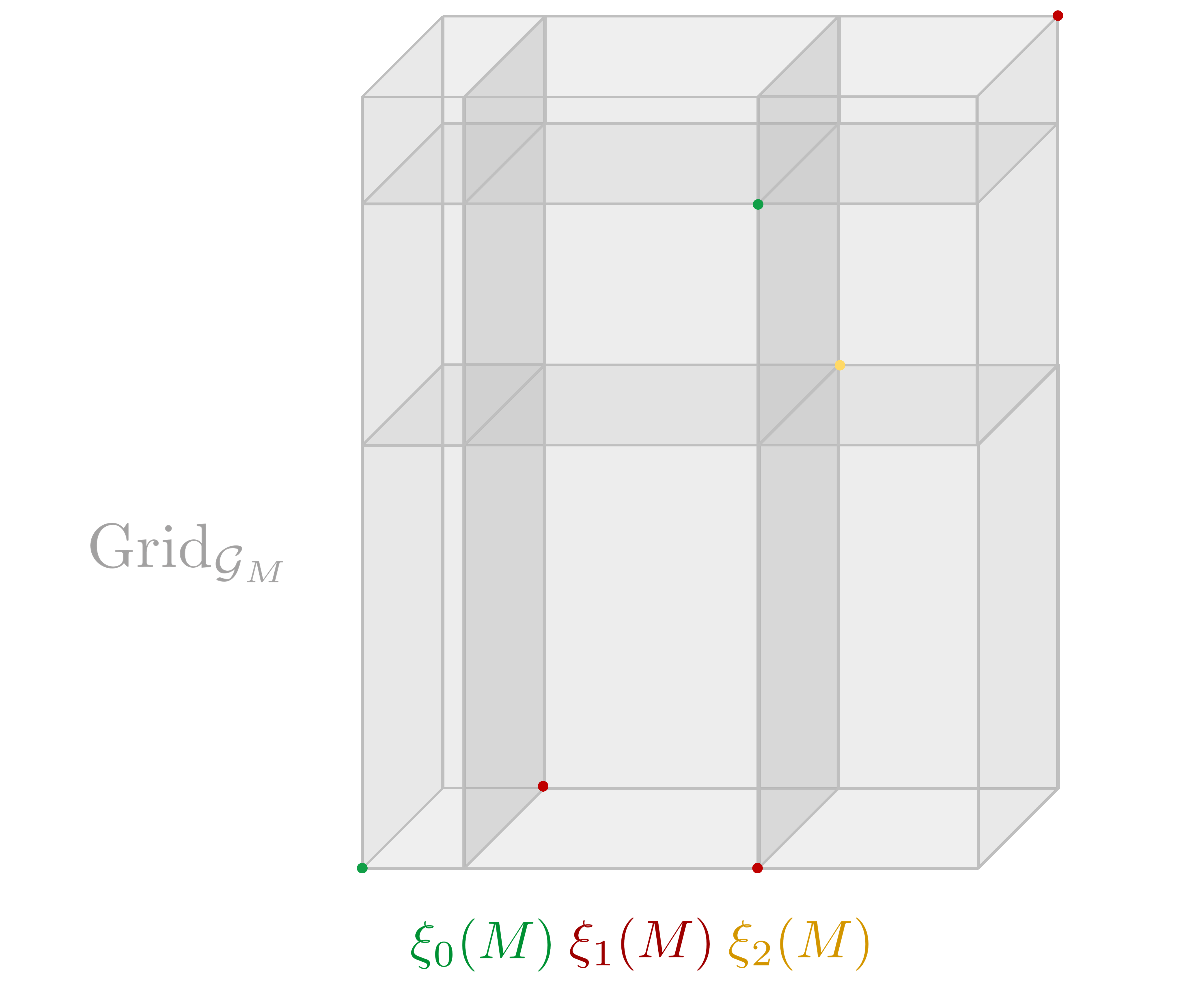} 
        \caption{Multiparameter Betti numbers of some 3-parameter module $M\in \vec{vect}_\text{fin}^{\vec{R}^3}$ together with the collection of axis aligned hyperplanes $\text{Grid}_{\mathcal{G}_M}$.}
    \end{subfigure}
    \caption{Multiparameter Betti numbers of multiparameter persistence modules together with their associated collection of axis aligned hyperplanes.}
    \label{fig:Multiparameter_Betti_Grids}
\end{figure}

\begin{defn}[Grid Function]
Let $[k]$ denote the set $\{1,...,k\}$. A \textcolor{Maroon}{grid function} is a map $\mathcal{G} = \Pi_{j=1}^n \mathcal{G}^j : \Pi_{j=1}^n [k_j]  \to \mathbb{R}^n$.
We shall use  \textcolor{Maroon}{$\text{Grid}_\mathcal{G}$} to denote the collection of axis aligned hyperplanes which passes through the image of this map:
$$ \text{Grid}_\mathcal{G} := \{\vec{a} \in \mathbb{R}^n : \pi_j(\vec{a}) \in \Ima \mathcal{G}^j \text{ for some } j\in [n]\}$$
\end{defn}

\begin{defn}[Controlling Constant]

For a grid function $\mathcal{G}: \Pi_{j=1}^n [k_j]  \to \mathbb{R}^n$ we define the \textcolor{Maroon}{controlling constant} $c(\mathcal{G})$ to be:

$$ c(\mathcal{G}) := \min \{\|\vec{a}-\vec{b}\|_\infty: \vec{a}\neq \vec{b}, \vec{a}, \vec{b} \in \Ima \mathcal{G}\} $$

In the edge cases for which $\Ima \mathcal{G}$ is empty or a singleton, we define the controlling constant to be infinite. 
\end{defn}

\begin{defn}[Multiparameter Betti Grid]

Suppose $M \in \fpmultimodules$ is a finitely presented multiparameter persistence module and that the multiparameter Betti numbers $\cup_i \xi_i(M)$ have $k_j$ distinct gradings in the $j^\text{th}$-coordinate.
Let $\mathcal{G}_M = \Pi_{j=1}^n \mathcal{G}^j_M :  \Pi_{j=1}^n [k_j] \to \mathbb{R}^n$ be the map such that $\cup_i \xi_i(M) \subset \Ima \mathcal{G}_M$. The image of this map shall be referred to as the  \textcolor{Maroon}{multiparameter Betti grid} of $M$ and denoted $\Ima \mathcal{G}_M$.

\end{defn}

The grid $\Grid$ marks the boundaries of the cells within which the module $M$ does not change.

For a finitely presented module $M \in \fpmultimodules$ we define the \textcolor{Maroon}{controlling constant} $c_M$ to be the controlling constant of the multiparameter Betti grid of $M$, $c_M = c(\mathcal{G}_M)$.
This is the minimum non-zero difference between coordinates of the collection of multiparameter Betti numbers $\cup_i \xi_i(M)$.

\begin{defn}[Complexity]

Suppose $M\in \fpmultimodules$ and $F_\bullet$ is a free resolution of $M$. We define the \textcolor{Maroon}{complexity} of the resolution $F_\bullet$ to be the cardinal $\mathcal{C}(F_\bullet) = \sum_{i} | F_i | $, where $|F_i|$ denotes the cardinality of the generating set of $F_i$. We define the \textcolor{Maroon}{complexity} of the persistence module $M$ to be the minimal complexity of a resolution of $M$, $\mathcal{C}(M) =\min_{F_\bullet \to M}\mathcal{C}(F_\bullet)$.

\end{defn}

Observe that the complexity of a module $M$ is realised by the complexity of a minimal resolution and thus $\mathcal{C}(M) =  |\cup_{i=0}^n \xi_i(M)|$.

\begin{lem}[Multiparameter Betti Grid is Determined by Generators and Relations]
\label{lem:xi0xi1detgrid}
Let $M\in \fpmultimodules$ be a finitely presented multiparameter persistence module. For all $j\geq 1$ if $\vec{s} \in \xi_{j+1}(M)$ then there exist $\vec{r}_i\in \xi_j(M)$ such that $\vec{s} = \bigvee_i \vec{r}_i$. Thus the the multiparameter Betti grid of $M$ is determined by $\xi_{0}(M) \cup \xi_{1}(M)$.
\end{lem}

\begin{proof}
We will just prove the result for $j=1$, since the result for $j>1$ follows identically. Let $F_\bullet$ be a minimal resolution of $M\in \fpmultimodules$. So that in particular $F_2 \twoheadrightarrow \ker(p_1) = \ker(F_1 \to F_0)$. Let $F_1$ have generating set $\{\vec{r}_i\}_{i = 1}^{k_1}$.
Suppose $p_2(\vec{s}) = \sum_i \mu_i \vec{x}^{\vec{s}-\vec{r}_i} \cdot \vec{r}_i \in \ker(p_1)$. Note that we must have $\vec{s}\geq \bigvee_{j: \mu_j \neq 0} \vec{r}_j $. Suppose $\vec{s}> \bigvee_{j: \mu_j \neq 0} \vec{r}_j$.  We observe that 
$$0 = p_1(p_2(\vec{s})) = p_1(\vec{x}^{\vec{s}- \bigvee_{j: \mu_j \neq 0} \vec{r}_j}\cdot \sum_i \mu_i \vec{x}^{\bigvee_{j: \mu_j \neq 0} \vec{r}_j-\vec{r}_i}\cdot \vec{r}_i) = \vec{x}^{\vec{s}- \bigvee_{j: \mu_j \neq 0} \vec{r}_j}\cdot p_1( \sum_i \mu_i \vec{x}^{\bigvee_{j: \mu_j \neq 0} \vec{r}_j-\vec{r}_i}\cdot \vec{r}_i)$$
Since $F_0$ is free it follows that 
$p_1( \sum_i \mu_i \vec{x}^{\bigvee_{j: \mu_j \neq 0} \vec{r}_j-\vec{r}_i}\cdot \vec{r}_i) = 0$ and hence there is some element 
${m}\in F_2$ with $ p_2({m}) = \sum_i \mu_i \vec{x}^{\bigvee_{j: \mu_j \neq 0} \vec{r}_j-\vec{r}_i}\cdot \vec{r}_i$. Thus 
$\vec{s}- \vec{x}^{\vec{s}-  \bigvee_{j: \mu_j \neq 0}\vec{r}_j}\cdot {m} \in \ker(p_2)$ and so $\vec{s}\not \in \xi_2(M)$ else we contradict minimality of the resolution $F_\bullet$.
\end{proof}

The following \Cref{lem:Cancellation} gives a simple criterion for reducing the complexity of a resolution.
Let $\langle \cdot , \cdot \rangle$ denote the point-wise vector space inner product of a finitely generated free module which makes it's generating set an orthonormal basis.

\begin{lem}[Cancellation Lemma]

Let $M\in \fpmultimodules$ be a finitely presented multiparameter persistence module. Suppose $F_\bullet \to M$ is a free resolution such that there exists a generator $\vec{b}$ of $F_{i}$, and a generator $\vec{r}_\vec{b}$ of $F_{i+1}$ with $\text{grade}(\vec{b})=\text{grade}(\vec{r}_\vec{b})$ and $\langle p_{i+1}(\vec{r}_\vec{b}), \vec{b}\rangle \neq 0$, then $F_\bullet$ is not a minimal resolution, and there is a free resolution $F_\bullet' \to M$ with $\xi_0(F'_{i+1}) = \xi_0(F_{i+1}) \setminus \{\vec{r}_\vec{b}\}$ and $\xi_0(F'_i) = \xi_0(F_i) \setminus \{\vec{b}\}$. That is to say we may cancel the pair of generators $\vec{b}$ and $\vec{r}_\vec{b}$ in the resolution.
\label{lem:Cancellation}
\end{lem}

\begin{proof}
Let $F'_\bullet$ denote the chain complex of free modules for which $\xi_0(F'_{i+1}) = \xi_0(F_{i+1}) \setminus \{\vec{r}_\vec{b}\}$, $\xi_0(F'_i) = \xi_0(F_i) \setminus \{\vec{b}\}$ and $\xi_0(F'_j) = \xi_0(F_j)$ otherwise, with morphisms inherited from $F_\bullet$.  Define chain maps $f: F_\bullet  \to F'_\bullet $ and $g:F'_\bullet \to F_\bullet $ to be the \textit{identity} on $\vec{a}\not\in\{ \vec{r}_\vec{b},\vec{b}\}$; ($f_i(\vec{a}) =  \vec{a} $, $g_i(\vec{a}) =\vec{a}$). Define $f_{i+1}(\vec{r}_\vec{b}) = 0$ and $f_i(\vec{b}) = \vec{b} - \frac{1}{\langle p_{i+1}(\vec{r}_\vec{b}), \vec{b}\rangle}p_{i+1}(\vec{r}_\vec{b})$.
Then we observe that $f\circ g = \id_{F'_\bullet}$ and, $g\circ f \simeq  \id_{F_\bullet}$ is realised by the homotopy $H:F_\bullet \to F_{\bullet+1} $ which is zero everywhere except $H_i(\vec{b}) = \frac{1}{\langle p_{i+1}(\vec{r}_\vec{b}),\vec{b}\rangle}\vec{r}_\vec{b}$. Indeed $p_{i+1}\circ H_i(\vec{b}) = \frac{1}{\langle p_{i+1}(\vec{r}_\vec{b}),\vec{b}\rangle}p_{i+1}(\vec{r}_\vec{b}) = (\id_{F_i} - g_i\circ f_i)(\vec{b})$ and $ H_{i+1}\circ p_{i+1}(\vec{r}_\vec{b}) = H_{i+1}(\langle p_{i+1}(\vec{r}_\vec{b}),\vec{b}\rangle \vec{b}) = \vec{r}_\vec{b} = (\id_{F_{i+1}} -  g_{i+1}\circ f_{i+1})(\vec{r}_\vec{b})$.
\end{proof}



\section{Merge and Simplification Functors}
\label{sec:M&SFunctors}

In this Section we define two families of endofunctors acting on $\fpmultimodules$. These functors will be used in the proof of our main result \Cref{thm:LocalEquivalence}. We shall show that these functors are well-defined and explicit how these functors affect the presentation of a module.

\subsubsection*{Notation}
Let us establish some notation conventions for this section. We shall denote free multigraded module resolutions as $(F_\bullet, p_\bullet)$ where $p_i : F_i \to F_{i-1}$. We shall denote generators of $F_0$ using the letter $\vec{b}$ for ``births" and we shall denote generators of $F_1$ using the letter $\vec{r}$ for ``relations". We shall simultaneously use $\vec{b}$ to denote both the generator (an element of a graded set) and its grading. If there is ambiguity we shall explicitly write $\text{gr}(\vec{b})$ to denote the grading. Given an element $\vec{a}\in \mathbb{R}^n$ we will denote the slope $\vec{1}$ line through $\vec{a}$ as $\mathcal{L}_\vec{a}$ with $\mathcal{L}_\vec{a}(0)=\vec{a}$ and the isometric embedding with respect to $\|\cdot \|_\infty$. For a set $ S \subset \mathbb{R}^n$ and $\vec{a} \in \mathbb{R}^n$ we use the shorthand $\|\vec{a} - S \|_\infty:= \inf_{\vec{s}\in S}\|\vec{a} - \vec{s}\|_\infty$.

\subsection{Merge Functor}

The first family of functors we define we call merge functors. These are analogous to the merge operations defined in \cite{carriere_structure_2018} for Reeb graphs which are heavily used in \cite{carriere_local_2017}. \Cref{fig:MergedModule} gives a pictorial description of the action of a merge functor on a multiparameter module. The merge functor moves the grades of multiparameter Betti numbers lying close to a grid to lie on that grid. Moving the multiparameter Betti numbers may cause cancellations (\Cref{lem:Cancellation}).  

We give a module-theoretic and category-theoretic formulation of the merge functor and establish that these formulations are equivalent. The module-theoretic formulation clarifies how the functor changes the presentation of a module, whilst the category-theoretic formulation establishes the functoriality and exactness of the merge functor immediately.
 
\begin{figure}
\centering
\begin{subfigure}{.7\textwidth}
  \centering
  \includegraphics[width=0.9\linewidth]{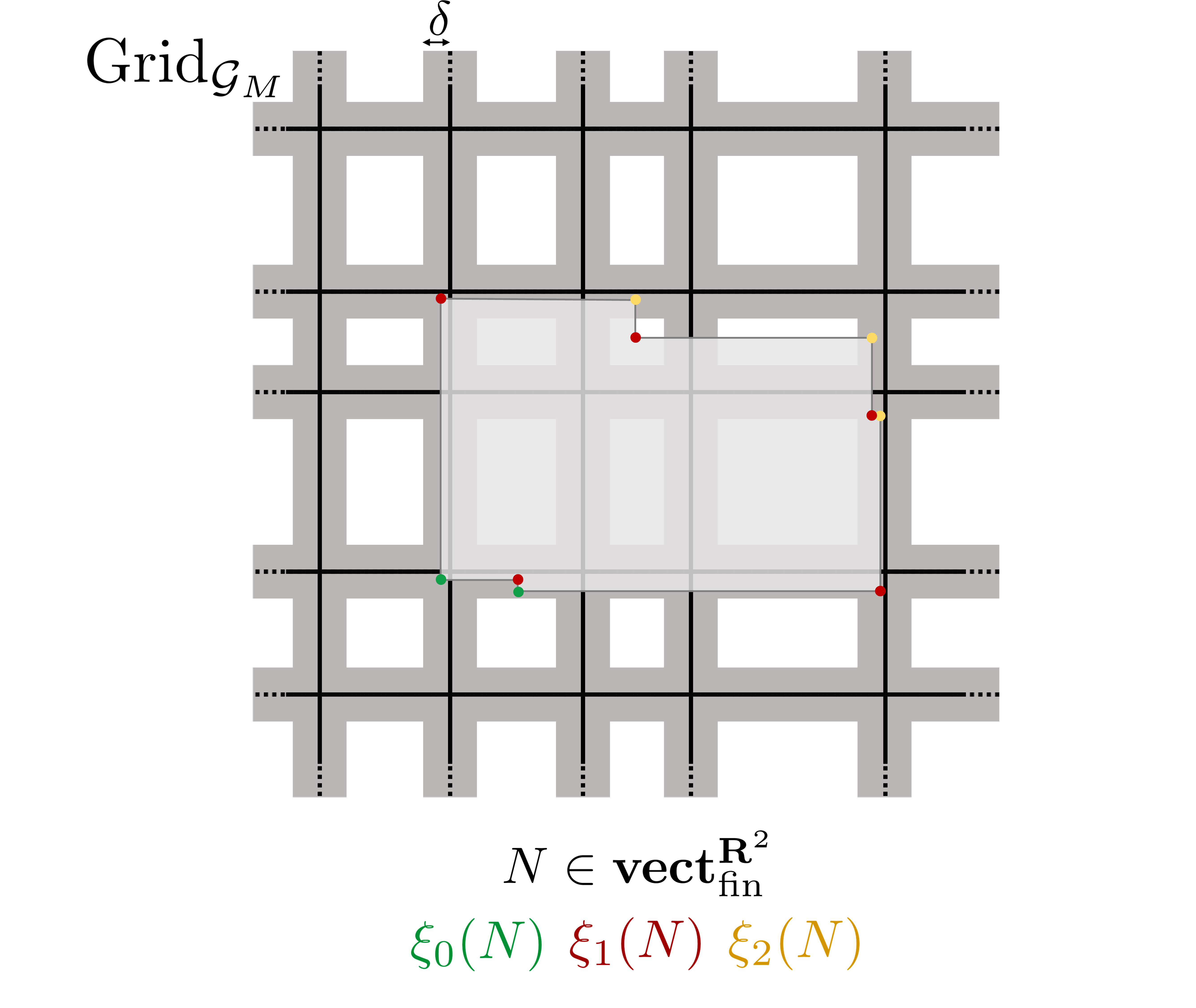}
  \caption{}
  \label{subfig:UnmergedModule}
\end{subfigure}%
\\
\begin{subfigure}{.7\textwidth}
  \centering
  \includegraphics[width=0.9\linewidth]{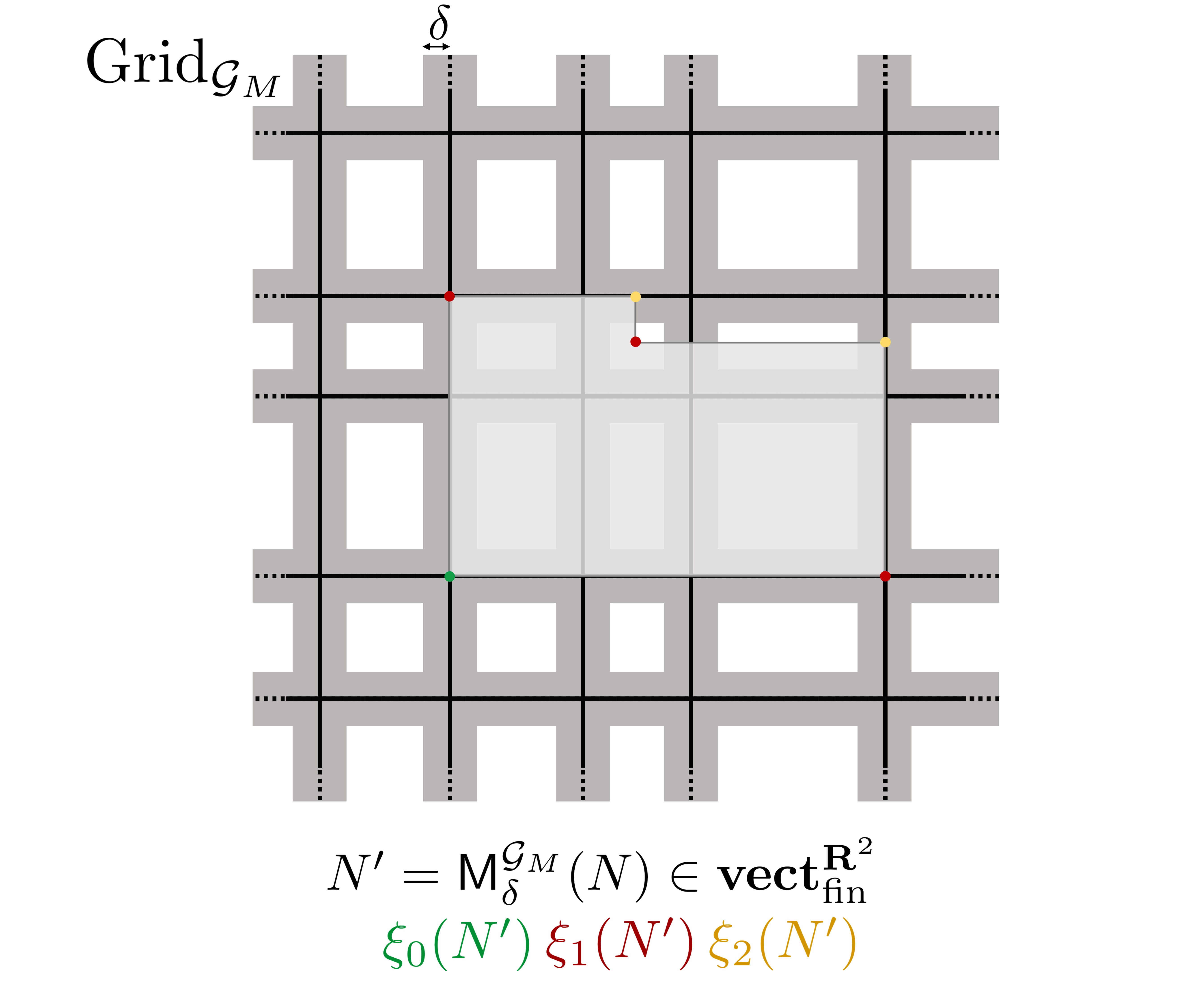}
  \caption{}
  \label{subfig:MergedModule}
\end{subfigure}
\caption{An example of a module under a Merge functor. In (a) we depict an interval module $N$ overlayed on the $\delta$-neighbourhood of the grid $\Grid$, with the multiparameter Betti numbers of $N$ marked. We observe cancellations of the multiparameter Betti numbers under the Merge functor $\textsf{M}_\delta^{\mathcal{G}_M}$, and the multiparameter Betti numbers of $N$ lying $\delta$-close to the grid are merged to lie on the grid.}
\label{fig:MergedModule}
\end{figure}

\begin{defn}[Merge Functions]

Let $\mathcal{G}: [k] \to \mathbb{R}$ be a finite grid with controlling constant $c$, and let $\delta < \frac{c}{2}$. Define the \textcolor{Maroon}{merge functions} $\Merge: \mathbb{R} \to \mathbb{R}$, $\textsf{M}_\delta^{\mathcal{G}_+}: \mathbb{R} \to \mathbb{R}$, $\textsf{M}_{\delta}^{\mathcal{G}_-}  : \mathbb{R} \to \mathbb{R}$ as follows:

\begin{align*} 
\Merge(x) &= 
        \begin{cases}
                        \mathcal{G}(i) & \text{ if } x \in [\mathcal{G}(i) - \delta,\mathcal{G}(i) + \delta]  \\
                        x & \text{ otherwise}
        \end{cases} \\
 \textsf{M}_\delta^{\mathcal{G}_+}(x) &= 
        \begin{cases}
                        \mathcal{G}(i) & \text{ if } x \in [\mathcal{G}(i) - \delta,\mathcal{G}(i)]  \\
                        x & \text{ otherwise}
        \end{cases} \\
\textsf{M}_{\delta}^{\mathcal{G}_-}(x) &= 
        \begin{cases}
                        \mathcal{G}(i) & \text{ if } x \in [\mathcal{G}(i)
                        ,\mathcal{G}(i) + \delta]  \\
                        x & \text{ otherwise}
        \end{cases} 
\end{align*}

\noindent Note that $\Merge = \textsf{M}_{\delta}^{\mathcal{G}_-}\circ \textsf{M}_\delta^{\mathcal{G}_+}$ and the merge functions are projections.

For a multiparameter grid function $\mathcal{G} = \mathcal{G}^1\times ... \times\mathcal{G}^n $ with controlling constant $c$ and $\delta < \frac{c}{2}$. The function $\Merge : \mathbb{R}^n \to \mathbb{R}^n$ is defined using the merge functions $\textsf{M}_{\delta}^{\mathcal{G}_i}$ coordinate-wise, and similarly for $\textsf{M}_\delta^{\mathcal{G}_+}$ and $\textsf{M}_{\delta}^{\mathcal{G}_-}$

\end{defn}

Note that a merge function preserves the partial order of a pair of elements in $\mathbb{R}^n$. That is to say $\vec{a}\leq \vec{b}$ implies  $\Merge(\vec{a})\leq \Merge(\vec{b})$. We first define the action of the merge functor on free modules.

\begin{defn}[Merge Functor for Free Modules]

Let $\mathcal{G}= \Pi_{i=1}^n \mathcal{G}^i: \Pi_{i=1}^n [k_i] \to \mathbb{R}^n$ be a grid with separation constant $c$, and let $\delta < \frac{c}{2}$. We define the $\delta$-merge of the free module $F = \text{Free}[\mathcal{X}]$ with respect to the grid $\mathcal{G}$ to be the free module with grading set regraded via the merge function:  \textcolor{Maroon}{$\Merge(F) = \text{Free}[\Merge(\mathcal{X})]$}, together with the obvious action on morphisms of free modules. 
\end{defn}

Since the merge function preserves the partial order on $\mathbb{R}^n$, $\Merge$ is well defined on morphisms of free modules. It is straight forward to check that $\Merge$ respects identity morphisms for free modules and compositions of morphisms between free modules.

\begin{defn}[Merge Functor]

Let $M\in \fpmultimodules$ and $\mathcal{G}= \Pi_{i=1}^n \mathcal{G}^i: \Pi_{i=1}^n [k_i] \to \mathbb{R}^n$ a grid with separation constant $c$, and let $\delta < \frac{c}{2}$. Suppose $F_1 \to F_0 \to M$ is a presentation of $M$. We define the $\delta$-merge of the module $M$ with respect to the grid $\mathcal{G}$ to be the module \textcolor{Maroon}{$\text{coker } \Merge(F_1 \to F_0)$} i.e. the module $\Merge(M)$ with presentation:
$$\Merge(F_1) \to \Merge(F_0) \to  \Merge(M)$$
We define $\Merge$ on a morphism of multiparameter persistence modules $f:M\to N$ to be the morphism induced from $\Merge$ applied to the free module morphism $f_0$, (the lift of $f$ between the free modules generating $M$ and $N$).
\end{defn}

\begin{prop}[Merge Functor is Well-defined]

$\Merge$ is a well-defined endofunctor on the category $\fpmultimodules$ for all $\delta < \frac{c}{2}$.
\end{prop}

\begin{proof}
We first observe that $\Merge$ is independent of the presentation chosen. Any presentation gives rise to a free resolution and two such free resolutions of a module are homotopy equivalent. Since $\Merge$ respects compositions and identity maps of free modules the free resolutions under the map $\Merge$ remain homotopy equivalent.

Let us now show that $\Merge$ is well defined on morphisms. Suppose $f:M\to N$ is a morphism of modules, and suppose $F_1 \to F_0 \to M$ is a presentation of $M$ and $G_1 \to G_0 \to N$ is a presentation of $N$. Note that $f$ lifts to a chain map between these presentations:
\[
\begin{tikzcd}
F_1 \arrow[d, "f_1"] \arrow[r, "p_1"] & F_0 \arrow[d, "f_0"] \arrow[r, "p_0"] & M \arrow[d, "f"]
\\
G_1  \arrow[r, "q_1"] & G_0  \arrow[r, "q_0"] & N
\end{tikzcd}
\]
Let $F_0 = \text{Free}(\{\vec{b}^M_i\})$, $G_0 = \text{Free}(\{\vec{b}^N_j\})$, $F_1 = \text{Free}(\{\vec{r}^M_l\})$, $G_1 = \text{Free}(\{\vec{r}^N_m\})$, and recall that $f$ is determined by the action of $f_0$ on the generators $\{\vec{b}^M_i\}$. Suppose that  $f_0(\vec{b}^M_i) = \sum_j c_{i,j}\vec{x}^{\vec{b}^M_i- \vec{b}^N_j}\cdot\vec{b}^N_j $
The morphism $\Merge(f): \Merge(M) \to \Merge(N)$ is defined to act on generators as:
$$\Merge(f)(\Merge(\vec{b}^M_i)) = \sum_j c_{i,j}\vec{x}^{\Merge(\vec{b}^M_i)- \Merge(\vec{b}^N_j)}\cdot\Merge(\vec{b}^N_j)$$




To check $\Merge(f)$ is well defined for arbitrary $f$ it suffices to show that $\Merge(f_0\circ p_1)(\Merge(\vec{r}^M_l))\in \Ima \Merge(q_1)$. 
As $f$ is well defined we have that $f_0\circ p_1(\vec{r}^M_l)\in \Ima q_1$:
\begin{align*}
    f_0\circ p_1(\vec{r}^M_l) &= f_0(\sum_i c_{l,i} \vec{x}^{\vec{r}^M_l- \vec{b}^M_i}\cdot\vec{b}^M_i ) = \sum_j \sum_i c_{i,j}c_{l,i} \vec{x}^{\vec{r}^M_l- \vec{b}^N_j}\cdot\vec{b}^N_j \\&= q_1 (\sum_m a_m \vec{x}^{\vec{r}^M_l - \vec{r}^N_m} \cdot \vec{r}^N_m) = \sum_j\sum_m a_m c_{m,j} \vec{x}^{\vec{r}^N_m- \vec{b}^N_j}\cdot\vec{b}^N_j
\end{align*}

\noindent For some scalars $a_m$. It follows that:
\begin{align*}
\Merge(f_0\circ p_1)(\Merge(\vec{r}^M_l)) &= \sum_j\sum_m a_m c_{m,j} \vec{x}^{\Merge(\vec{r}^N_m)- \Merge(\vec{b}^N_j)}\cdot\Merge(\vec{b}^N_j)
\\
&=  \Merge(q_1) (\sum_m a_m \vec{x}^{\Merge(\vec{r}^M_l) - \Merge(\vec{r}^N_m)} \cdot \Merge(\vec{r}^N_m))
\end{align*}
and hence $\Merge(f)$ is well defined.


\end{proof}

Another important property of the family of merge endofunctors is that $\Merge$ perturbs multiparameter persistence modules by no more than $\delta$ in the interleaving distance. 

\begin{prop}[$d_I(\Merge(M),M) \leq \delta$]

For any grid $\mathcal{G}$ and $\delta< \frac{c(\mathcal{G})}{2}$, if $M \in \fpmultimodules$ then $d_I(\Merge(M),M) \leq \delta$.
\end{prop}

\begin{proof}
We will show that $d_I(\Merge(M),M) \leq \delta$ via an explicit interleaving.
Let $F_1\to_{p_1} F_0\to_{p_0} M$ be a presentation of $M$ where $F_0$ is free on the set $\{\vec{b}_i\}$ and $F_1$ is free on the set $\{\vec{r}_j \}$. Consider the maps $f: M \to \Merge(M)T_\delta$ and $g :\Merge(M) \to MT_\delta$ defined on generators by:

$$ f(\vec{b}_i) = \vec{x}^{\delta\vec{1}+ (\vec{b}_i- \Merge(\vec{b}_i))} \cdot \Merge(\vec{b}_i), \ g(\Merge(\vec{b}_i)) = \vec{x}^{\delta\vec{1}+(( \Merge(\vec{b}_i)- \vec{b}_i)} \cdot \vec{b}_i $$

If these morphisms are well-defined it is clear they comprise a $\delta$-interleaving since their composition yields $g\circ f(\vec{b}_i) = \vec{x}^{2\delta \vec{1}}\cdot \vec{b}_i $ and $f\circ g(\Merge(\vec{b}_i)) = \vec{x}^{2\delta \vec{1}}\cdot \Merge(\vec{b}_i) $. 
Observe that $\delta\vec{1}\geq  \vec{b}_i- \Merge(\vec{b}_i) \geq -\delta\vec{1}$ and so it suffices to check that the proposed morphisms respect relations. If 
$\sum_i \alpha_i \vec{x}^{\vec{s}-\vec{b}_i} \vec{b}_i \in \Ima(p_1)$
then 
$\sum_i \alpha_i \vec{x}^{\vec{s}-\vec{b}_i} \vec{b}_i = \sum_j \lambda_j \vec{x}^{\vec{s}-\vec{r}_j} p_1(\vec{r}_j)$ for some scalars $\lambda_j$. 
We observe that:
\begin{align*}
f(\sum_i \alpha_i \vec{x}^{\vec{s}-\vec{b}_i} \vec{b}_i) &=
\sum_i \alpha_i \vec{x}^{\vec{s}-\vec{b}_i+\delta\vec{1}+ (\vec{b}_i- \Merge(\vec{b}_i))} \cdot \Merge(\vec{b}_i) \\&=
\sum_i \alpha_i \vec{x}^{\vec{s}+\delta\vec{1}- \Merge(\vec{b}_i)} \cdot \Merge(\vec{b}_i) 
\\ &= \sum_j \lambda_j \vec{x}^{\vec{s}+\delta\vec{1}- \Merge(\vec{r}_j)} \cdot \Merge(p_1)( \Merge(\vec{r}_j)) \in \Ima(\Merge(p_1))
\end{align*}
Similarly, if $\sum_i \alpha_i \vec{x}^{\vec{s}-\Merge(\vec{b}_i)} \cdot \Merge(\vec{b}_i) \in \Ima(\Merge(p_1))$ 
then 
$\sum_i \alpha_i \vec{x}^{\vec{s}-\Merge(\vec{b}_i)}   \cdot \Merge(\vec{b}_i) = \sum_j \lambda_j \vec{x}^{\vec{s}-\Merge(\vec{r}_j)}  \cdot \Merge(p_1)(\Merge(\vec{r}_j))$. We observe that:

\begin{align*}
g(\sum_i \alpha_i \vec{x}^{\vec{s}-\Merge(\vec{b}_i)} \cdot \Merge(\vec{b}_i)) &=
\sum_i \alpha_i \vec{x}^{\vec{s}-\Merge(\vec{b}_i)+\delta\vec{1}+ ( \Merge(\vec{b}_i)- \vec{b}_i)} \cdot \vec{b}_i \\&=
\sum_i \alpha_i \vec{x}^{\vec{s}+\delta\vec{1}- \vec{b}_i} \cdot \vec{b}_i 
\\ &= \sum_j \lambda_j \vec{x}^{\vec{s}+\delta\vec{1}- \vec{r}_j} \cdot p_1( \vec{r}_j) \in \Ima(p_1)
\end{align*}

Hence the proposed interleaving morphisms are well-defined and exhibit that $d_I(\Merge(M),M) \leq \delta$.


\end{proof}

One can directly show exactness of the merge functors using our modules theoretic definition, but this follows more naturally from a category-theoretic perspective of the merge functors.

\begin{prop}[Merge Functors as Kan Extensions]

Let $\mathcal{G}$ be a grid and $\delta <\frac{c(\mathcal{G})}{2}$. The merge endofunctors $\textsf{M}_\delta^{\mathcal{G}_+}$ and $\textsf{M}_\delta^{\mathcal{G}_-}$ acting on $\fpmultimodules$ can be realised as the following left and right Kan extensions respectively:

\[
\begin{tikzcd}
\vec{A_+} \arrow[r, "M\circ \iota_+"] \arrow[d,"\iota_+"] &\vec{vect} \\
\vec{R}^n \ar[ur,dashed,"\textsf{M}_\delta^{\mathcal{G}_+}(M)" {rotate = 35, anchor = north}]
\end{tikzcd}
\hspace{3em}
\begin{tikzcd}
\vec{A_-} \arrow[r, "M\circ \iota_-"] \arrow[d,"\iota_-"] &\vec{vect} \\
\vec{R}^n \ar[ur,dashed,"{\textsf{M}_\delta^{\mathcal{G}_-}(M)}" {rotate = 35, anchor = north}] 
\end{tikzcd}
\]

Where $A_+ := \mathbb{R}^n \setminus \{\textsf{M}_\delta^{\mathcal{G}_+}(\vec{a}) \neq \vec{a}\}$ and $A_- := \mathbb{R}^n \setminus (\{\textsf{M}_\delta^{\mathcal{G}_-}(\vec{a}) \neq \vec{a}\} \cup \text{Grid}_\mathcal{G})$ equipped with inclusions $\iota_+$ and $\iota_-$ into $\mathbb{R}^n$.

\label{prop:MergeAreKanExt}
\end{prop}

\begin{proof}
As in the statement of the proposition let $A_+ := \mathbb{R}^n \setminus \{\textsf{M}_\delta^{\mathcal{G}_+}(\vec{a}) \neq \vec{a}\}$ and $A_- := \mathbb{R}^n \setminus (\{\textsf{M}_\delta^{\mathcal{G}_-}(\vec{a}) \neq \vec{a}\} \cup \text{Grid}_\mathcal{G})$ equipped with inclusions $\iota_+$ and $\iota_-$ into $\mathbb{R}^n$. Denote the left Kan extension of $M_+ = M\circ \iota_+$ along $\iota_+$ by $\text{Lan}_{\iota_+}(M_+)$ and the right Kan extension of $M_- = M \circ \iota_-$ along $\iota_-$ by $\text{Ran}_{\iota_-}(M_-)$:

\[
\begin{tikzcd}
\vec{A_+} \arrow[r, "M_+"] \arrow[d,"\iota_+"] &\vec{vect} \\
\vec{R}^n \ar[ur,dashed,"\text{Lan}_{\iota_+}(M_+)" {rotate = 35, anchor = north}]
\end{tikzcd}
\hspace{3em}
\begin{tikzcd}
\vec{A_-} \arrow[r, "M_-"] \arrow[d,"\iota_-"] &\vec{vect} \\
\vec{R}^n \ar[ur,dashed,"{\text{Ran}_{\iota_-}(M_-)}" {rotate = 35, anchor = north}] 
\end{tikzcd}
\]
These Kan extensions exist since $A_+$ and $A_-$ are small and $\vec{Vect}$ is bicomplete. It transpires that the Kan extensions remain pointwise finite dimensional. More precisely, we can realise the Kan extensions as the following limits:

$$ \text{Lan}_{\iota_+}(M_+)(\vec{a}) = \varinjlim_{ \iota_+(\vec{b})\to \vec{a}} M_+(\vec{b}) $$

$$ \text{Ran}_{\iota_-}(M_-)(\vec{a}) = \varprojlim_{ \vec{a} \to \iota_-(\vec{b})} M_-(\vec{b})  $$

Since $M$ is finitely presented, for all $\vec{a} \in \mathbb{R}^n$ there exists an element $\floor{\vec{a}}_{A_+} \in A_+$ such that $\floor{\vec{a}}_{A_+} \leq \vec{a}$, $\|\vec{a} - \floor{\vec{a}}_{A_+}\|_\infty  \leq \delta$ and for all $\vec{b}\in A_+$ with $\floor{\vec{a}}_{A_+}\leq \vec{b} \leq \vec{a}$ the morphism $M_+(\floor{\vec{a}}_{A_+}\leq \vec{b})$ is an isomorphism.
Similarly, for all $\vec{a} \in \mathbb{R}^n$ there exists an element $\ceil{\vec{a}}^{A_-} \in A_-$ such that $\ceil{\vec{a}}^{A_-} \geq \vec{a}$ , $\|\ceil{\vec{a}}^{A_-}- \vec{a} \|_\infty  \leq \delta$ and for all $\vec{b}\in A_-$ with $\ceil{\vec{a}}^{A_-}\geq \vec{b} \geq \vec{a}$ the morphism $M_+(\vec{b}\leq \ceil{\vec{a}}^{A_-})$ is an isomorphism. Thus we realise for all $\vec{a} \in \mathbb{R}^n$:

$$ \text{Lan}_{\iota_+}(M_+)(\vec{a}) \cong M_+(\floor{\vec{a}}_{A_+}) \hspace{1cm} \text{Ran}_{\iota_-}(M_-)(\vec{a}) \cong M_-(\ceil{\vec{a}}^{A_-}) $$

We will show that $\textsf{M}_\delta^{\mathcal{G}_+}(M)(\vec{a})$ is naturally isomorphic to $M_+(\floor{\vec{a}}_{A_+})$ for all $\vec{a} \in \mathbb{R}^n$:

If $\vec{a} \in A_+$ then $M_+(\floor{\vec{a}}_{A_+}\leq \vec{a})$ is an isomorphism. We also know that the grades of generators in the free resolution which are less than $\vec{a}$ remain less than $\vec{a}$ under the merge function $\textsf{M}_\delta^{\mathcal{G}_+}$ and no generator which was strictly greater than $\vec{a}$ becomes less than $\vec{a}$ under the merge function thus $M(\vec{a}) \cong \textsf{M}_\delta^{\mathcal{G}_+}(M)(\vec{a})$. Hence we have a sequence of isomorphisms:
$$\text{Lan}_{\iota_+}(M_+)(\vec{a}) \cong M_+(\floor{\vec{a}}_{A_+}) \cong M_+(\vec{a}) \cong  M(\vec{a}) \cong  \textsf{M}_\delta^{\mathcal{G}_+}(M)(\vec{a}) $$
Alternatively if $\vec{a} \notin A_+$ then observe that $\textsf{M}_\delta^{\mathcal{G}_+}(M)(\floor{\vec{a}}_{A_+} \leq \vec{a})$ is an isomorphism and hence we have a sequence of isomorphisms:

$$ \textsf{M}_\delta^{\mathcal{G}_+}(M)(\vec{a}) \cong \textsf{M}_\delta^{\mathcal{G}_+}(M)(\floor{\vec{a}}_{A_+}) \cong M(\floor{\vec{a}}_{A_+}) \cong M_+(\floor{\vec{a}}_{A_+}) \cong  \text{Lan}_{\iota_+}(M_+)(\vec{a})$$

Thus for all $M\in \fpmultimodules$ we have shown that $\textsf{M}_\delta^{\mathcal{G}_+}(M) \cong \text{Lan}_{\iota_+}(M\circ \iota_+)$. Showing that $\textsf{M}_\delta^{\mathcal{G}_-}(M) \cong \text{Ran}_{\iota_-}(M\circ \iota_-)$ proceeds as the dual of the argument above mutatis mutandis.
\end{proof}

Filtered (co)limits are exact in $\vec{vect}$. A quick check establishes that a sequence of multiparameter persistence modules is exact if and only if it is exact pointwise. Thus \Cref{prop:MergeAreKanExt} yields functoriality, well-definedness and exactness of the merge functor when realised as the following composition: $$\Merge(M) = \text{Ran}_{\iota_-}(\iota_-^\ast(\text{Lan}_{\iota_+}(\iota~^\ast_+(M))). $$

A consequence of exactness is that, for any free resolution of $M$ there is a free resolution of $\Merge(M)$ of the same complexity. Hence a merge functor never increases the complexity of a multiparameter persistence module.

\subsection{Simplification Functor}

The second family of functors we define simplify a module by removing ``$\varepsilon$-small features" of the module. A pictorial description of a simplification functor applied to an interval decomposable 2-parameter persistence module is depicted in \Cref{fig:Simplification}.

\begin{figure}[t]
    \centering
    \begin{subfigure}{0.3\textwidth}
        \centering
        \includegraphics[width=1\textwidth]{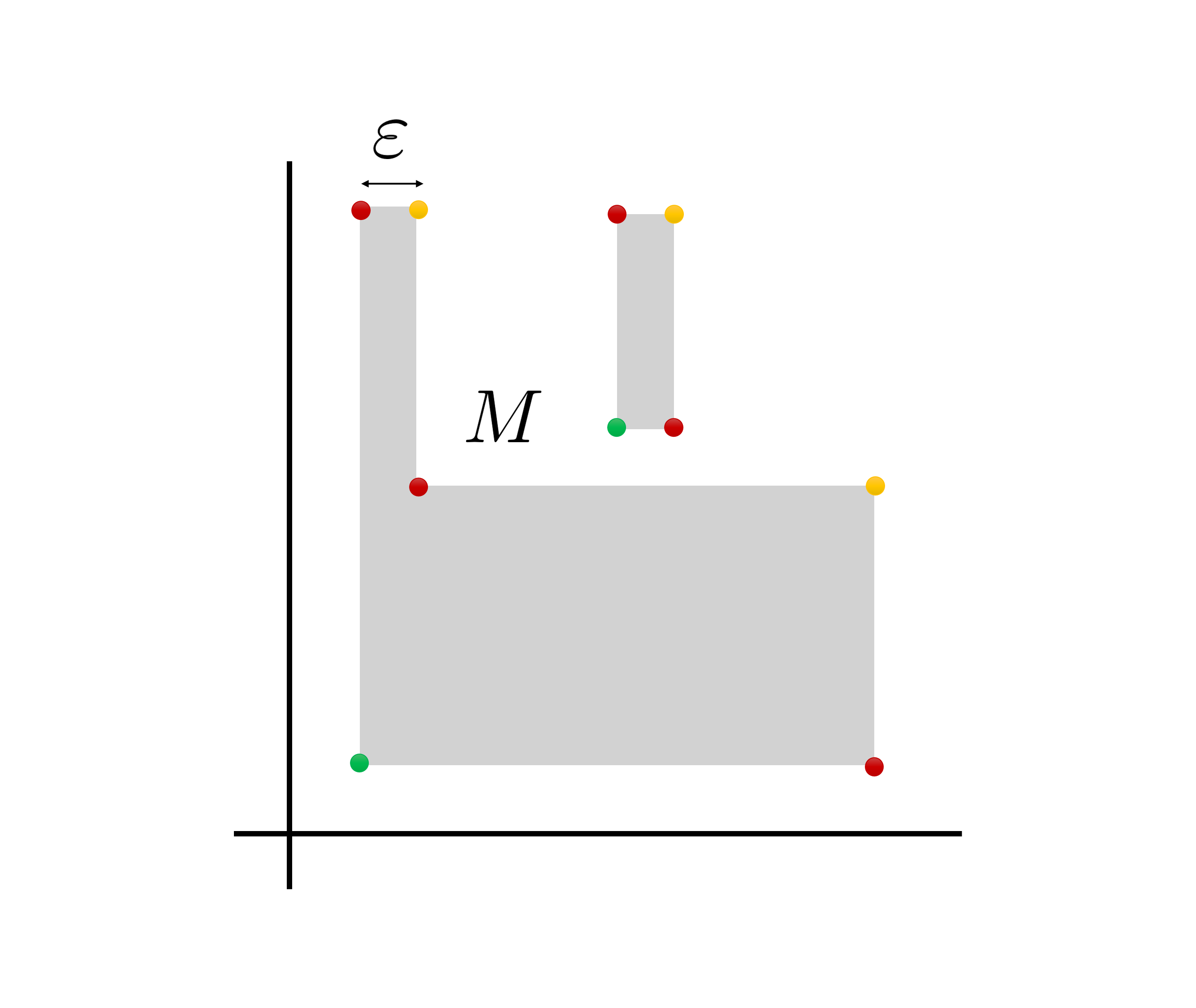} 
        \caption{A 2-parameter interval decomposable module $M$.}
    \end{subfigure}\hfill
    \begin{subfigure}{0.3\textwidth}
        \centering
        \includegraphics[width=1\textwidth]{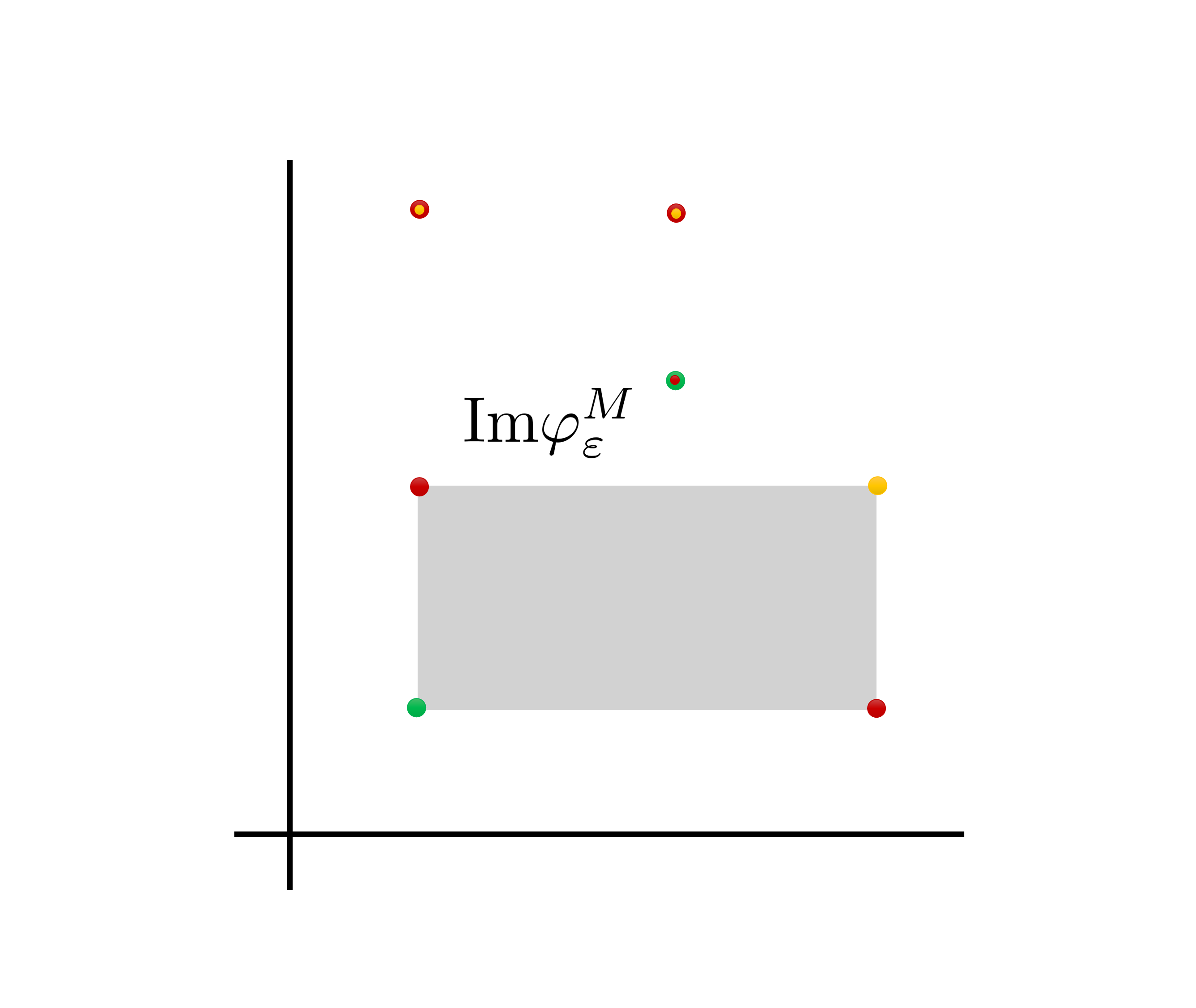} 
        \caption{The image of the internal morphism $\varphi_\varepsilon^M$.}
    \end{subfigure}\hfill
    \begin{subfigure}{0.3\textwidth}
        \centering
        \includegraphics[width=1\textwidth]{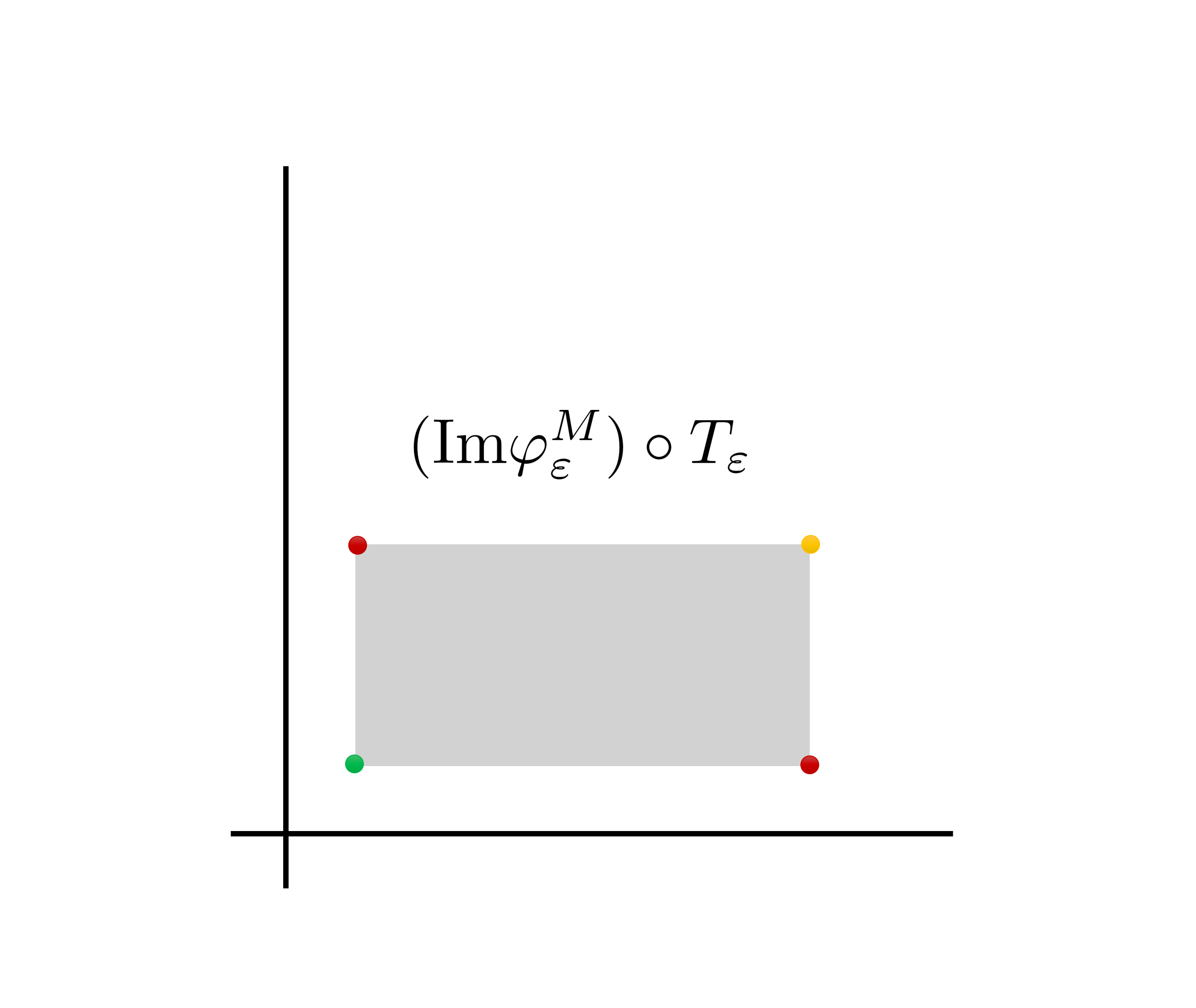} 
        \caption{The $\varepsilon$-simplified module $\textsf{S}_\varepsilon(M) = (\Ima \varphi_\varepsilon^M)\circ T_{\varepsilon}$.}
    \end{subfigure}
    \caption{A 2-parameter interval decomposable module $M$ under the $\varepsilon$-simplification functor $\textsf{S}_\varepsilon$ (\Cref{defn:SimplificationFunctor}).}
    \label{fig:Simplification}
\end{figure}

\begin{defn}[Simplification Functor]
\label{defn:SimplificationFunctor}
If $M \in \fpmultimodules$ define the $\varepsilon$-simplification of $M$, denoted \textcolor{Maroon}{$\textsf{S}_\varepsilon(M)$}, to be the shifted image under the internal translation $\varphi_\varepsilon^M$ of $M$. 
$$\textsf{S}_\varepsilon(M) = (\Ima \varphi_\varepsilon^M)\circ T_{\varepsilon}$$
\end{defn}

\begin{prop}[$d_I(M,\textsf{S}_{\varepsilon}(M) )\leq \varepsilon$]

$\textsf{S}_{\varepsilon}$ is an endofunctor of $\multimodules$ and perturbs modules by no more than $\varepsilon$ in the interleaving distance, that is $d_I(M,\textsf{S}_{\varepsilon}(M) )\leq \varepsilon$.
\end{prop}

\begin{proof}
The translation functor $T_\varepsilon^\ast$ is clearly a functor and the internal translation is the action of the element $\vec{x}^{\varepsilon\vec{1}} \in P_n$ on $P_n$-\textbf{Mod}.

We can explicitly define the interleaving morphisms realising $d_I(M,\textsf{S}_{\varepsilon}(M) )\leq \varepsilon$:
$f: M \Rightarrow \textsf{S}_{\varepsilon}(M)\circ T_\varepsilon$, $g: \textsf{S}_{\varepsilon}(M) \Rightarrow M\circ T_\varepsilon$.
Taking $f = \varphi_{2\varepsilon}^M$ and $g = \Ima \varphi_\varepsilon^M \hookrightarrow M\circ T_\varepsilon$ we see that $T^\ast_\varepsilon(g)\circ f =  \varphi_{2\varepsilon}^M$ and $T^\ast_\varepsilon(f)\circ g = \varphi_{2\varepsilon}^M \circ T_\varepsilon = \varphi_{2\varepsilon}^{\textsf{S}_{\varepsilon}(M)}$

\end{proof}

We explicit the effect of the simplification functor on a multiparameter presentation and on the barcode of a single parameter module.

\begin{lem}[Presentation Change under Internal Translation]
Suppose $F_\bullet$ is a free resolution of $M\in \fpmultimodules$ and let $F_1$ and $F_0$ have generating sets $\{\vec{r}_i\}$ and $\{\vec{b}_j\}$ respectively. There is a corresponding resolution $F_\bullet'$ of $\Ima \varphi_\varepsilon^M$ with a bijective correspondence of generating elements $\vec{b} \in \xi_0(F_0) \longleftrightarrow \vec{b}' \in \xi_0(F_0')$, $\vec{r} \in \xi_0(F_1) \longleftrightarrow \vec{r}' \in \xi_0(F_1') $  with the following grading shifts:
$$
 \text{gr}(\vec{b}') = \text{gr}(\vec{b}) + \vec{\varepsilon}
$$

$$
 \text{gr}(\vec{r}') = \text{gr}(\vec{r}) \vee \bigvee_{i : \lambda_i \neq 0}\text{gr}(\vec{b}_i')  \text{ where } p_1(\vec{r}) = \sum_i \lambda_i \vec{x}^{\vec{r}-\vec{b}_i} \cdot \vec{b}_i 
$$
together with the obvious inherited morphisms:
$$
p_0'(\vec{b}') = \vec{x}^{\vec{\varepsilon}} \cdot p_0(\vec{b})
$$
$$
p_1'(\vec{r}') = \sum_i \lambda_i \vec{x}^{\vec{r}'-\vec{b}'_i} \cdot \vec{b}'_i
$$
\label{prop:SmoothedResolution}
\end{lem}

\begin{proof}
Clearly $p_0' : F_0' \to \Ima\varphi_\varepsilon^M$ is surjective. For every $\vec{r}' \in \xi_0(F_1')$ we have that:
$$ p_0' \circ p_1'(\vec{r}') = p_0'(\sum_i \lambda_i \vec{x}^{\vec{r}'-\vec{b}'_i} \cdot \vec{b}'_i) = \sum_i \lambda_i \vec{x}^{\vec{r}'+ \vec{\varepsilon} -\vec{b}'_i} \cdot p_0(\vec{b}_i) = \vec{x}^{\vec{r}' - \vec{r}} \cdot \sum_i \lambda_i \vec{x}^{\vec{r}-\vec{b}_i} \cdot p_0(\vec{b}_i) = \vec{x}^{\vec{r}' - \vec{r}} \cdot p_0 \circ p_1 (\vec{r}) = 0   $$
and so $\Ima p_1' \subset \ker p_0'$. 
Moreover suppose that $\sum_i \lambda_i \vec{x}^{\vec{s}-\vec{b}_i'} \cdot \vec{b}_i' \in \ker p_0'$ then $\sum_i \lambda_i \vec{x}^{\vec{s}+\vec{\varepsilon}-\vec{b}_i} \cdot \vec{b}_i \in \ker p_0 = \Ima p_1$, hence there exists $\vec{r}_j\in \xi_0(F_1)$ and scalars $\mu_j$ such that $p_1(\sum_j \mu_j \vec{x}^\vec{\vec{s}-\vec{r}_j}\cdot \vec{r}_j) = \sum_i \lambda_i \vec{x}^{\vec{s}+\vec{\varepsilon}-\vec{b}_i} \cdot \vec{b}_i$ and thus $p_1'(\sum_j \mu_j \vec{x}^\vec{\vec{s}-\vec{r}_j'}\cdot \vec{r}_j') = \sum_i \lambda_i \vec{x}^{\vec{s}-\vec{b}_i'}\cdot \vec{b}_i' \in \Ima p_1'$. Hence we have shown exactness of $F_1' \to F_0' \to \Ima \varphi_\varepsilon^M \to 0$.
\end{proof}

\begin{lem}[Presentation Change under Simplification]
Suppose $F_\bullet$ is a free resolution of $M\in \fpmultimodules$ and let $F_1$ and $F_0$ have generating sets $\{\vec{r}_i\}$ and $\{\vec{b}_j\}$ respectively. There is a corresponding resolution $F_\bullet'$ of $\textsf{S}_\varepsilon(M)$ with a bijective correspondence of generating elements  $\vec{b} \in \xi_0(F_0) \longleftrightarrow \vec{b}' \in \xi_0(F_0')$, $\vec{r} \in \xi_0(F_1) \longleftrightarrow \vec{r}' \in \xi_0(F_1') $ with the following grading shifts:
$$ \text{gr}(\vec{b}') = \text{gr}(\vec{b})
$$

$$ \text{gr}(\vec{r}') = \left(\text{gr}(\vec{r}) \vee \bigvee_{i : \lambda_i \neq 0}\text{gr}(\vec{b}_i')\right) - \vec{\varepsilon} \text{ where } p_1(\vec{r}) = \sum_i \lambda_i \vec{x}^{\vec{r}-\vec{b}_i} \cdot \vec{b}_i 
$$
together with the obvious inherited morphisms.
\label{lem:SimplificationGenerators}
\end{lem}

\begin{lem}[Bar Code Simplification]
Let $M\in \vec{vect}^\vec{R}$ have interval decomposition $\bigoplus_{j\in \mathcal{J}} \mathds{1}^{I_j}$, where $I_j = [b_j,d_j)$. 
Define the $\varepsilon$-simplification of the intervals $I_j$ by:

\begin{align*} 
\textsf{S}_\varepsilon(I_j) =      
    \begin{cases}
            [b_j,\infty) & \text{ if } d_j = \infty  \\
            [b_j,d_j-\varepsilon)  & \text{ if } d_j-b_j>\varepsilon  \\
            \emptyset & \text{ otherwise.}
    \end{cases} 
\end{align*}

The module $\textsf{S}_\varepsilon(M)\cong \bigoplus_{j\in \mathcal{J}} \mathds{1}^{\textsf{S}_\varepsilon(I_j)}$.
\label{lem:BarcodeSimp}
\end{lem}

\begin{proof}
Consider the minimal free resolution of $M$ where $F_0$ is free on the graded set $\{b_j\}_{j\in \mathcal{J}}$, $F_1$ is free on the graded set $\{d_j\neq \infty\}_{j\in \mathcal{J}}$, and $p_1(d_j) = x^{d_j-b_j} \cdot b_j$. By \Cref{prop:SmoothedResolution}, $\textsf{S}_\varepsilon(M)$ has a corresponding resolution $F_\bullet'$. Moreover $F_0'$ is free on the graded set $\{b_j'\}_{j\in \mathcal{J}}$ with $\text{grade}(b_j') = b_j$, $F_1$ is free on the graded set $\{d'_j\}_{j\in \mathcal{J}}$ with $\text{grade}(d_j') = b_j \vee (d_j-\varepsilon)$, and $p_1'(d'_j) = x^{d'_j-b'_j} \cdot b'_j$. As claimed, we have $\text{coker}(p_1') \cong \bigoplus_{j\in \mathcal{J}} \mathds{1}^{\textsf{S}_\varepsilon(I_j)}$.
\end{proof}

The simplification functor naturally restricts to the 1-parameter submodules of a multiparameter persistence modules. For slope $\vec{1}$ lines $\mathcal{L}_\vec{a}$, the $\varepsilon$-simplification functor commutes with the restriction to the line $\mathcal{L}_\vec{a}$:
$$ \left.\textsf{S}_\varepsilon\right|_{\mathcal{L}_\vec{a}}(M^{\mathcal{L}_\vec{a}}) = \textsf{S}_\varepsilon(M)^{\mathcal{L}_\vec{a}}$$ 
Thus using \Cref{lem:BarcodeSimp} we can track how the simplification functor $\textsf{S}_\varepsilon$ applied to a multiparameter module affects bar codes in the fibered bar code of that module.

\section{Local Equivalence}
\label{sec:LocalEquivalence}

In this section we prove the main result of this article \Cref{thm:LocalEquivalence}.

We make use of the merge and simplification functors defined in \Cref{sec:M&SFunctors} and their effect on the presentation and fibered bar code of multiparameter persistence modules. After a series of technical lemmas we establish the following local equivalence result:

\begin{restatable}[Local Equivalence]{thm0}{localequiv}
Suppose $M,N \in \fpmultimodules$ are finitely presented multiparameter persistence modules, and $M$ has controlling constant $c_M= c(\mathcal{G}_M)$. For all $\kappa \in [0, \frac{1}{34})$, if $d_I(M,N)=\varepsilon$ and $ \varepsilon < \frac{c_M}{2(34\kappa  +1)} $ then the matching distance is bounded below $d_0(M,N) > \kapeps$.
\label{thm:LocalEquivalence}
\end{restatable}

 Let $ B(M,\frac{c_M}{4})\subset (\fpmultimodules,d_I)$ denote the open ball centred at $M$ of radius $\frac{c_M}{4}$. \Cref{thm:LocalEquivalence} states that for all $N \in B(M,\frac{c_M}{4})$ we have:

$$ \frac{1}{34}d_I(M,N) \leq d_0(M,N) \leq d_I(M,N) $$

In particular, since the interleaving distance is complete on $\fpmultimodules$, this result implies that the fibered bar code of $M$ is distinct from the fibered bar code of any $N \in B(M,\frac{c_M}{4})$, and thus is a locally complete invariant for $M$.

The proof of \Cref{thm:LocalEquivalence} is built using a series of auxiliary lemmas. The holistic idea for the proof of \Cref{thm:LocalEquivalence} follows the same structure as the proof of the local equivalence result in \cite{carriere_local_2017}. We assume two modules have a small matching distance in comparison to one of their controlling constants $d_0(M,N) = \kapeps \ll c_M$ and in comparison to their interleaving distance $d_0(M,N) = \kapeps \ll \varepsilon = d_I(M,N)$. We can use the small matching distance $d_0(M,N) = \kapeps \ll c_M$ to deduce that the multiparameter Betti numbers of $N$ either lie close to the grid $\Ima \mathcal{G}_M$ or can be cancelled with a nearby multiparameter Betti number. Using a series of merges and simplifications we find a nearby module $\Tilde{N}$, ($d_I(N,\Tilde{N}) \leq \text{const}\cdot \kapeps$), such that the multiparameter Betti numbers of $\Tilde{N}$ are contained in the grid $\Ima \mathcal{G}_M$. If a $\delta$-interleaving between $\Tilde{N}$ and $M$ exist for some $\delta < \frac{c_M}{2}$ then \Cref{lem:LesnickLemma} implies that $\Tilde{N}$ and $M$ are isomorphic. We use these results to derive a contradiction to the triangle inequality for the interleaving distance when $\kappa$ is too small.

Let us first recall the push function \Cref{defn:pushfunction}. Note that the fibre of the push map $\text{push}_\mathcal{L}$ for an element $\vec{a}\in \mathcal{L}$ is the boundary of the downset of $\vec{a}$:

$$ \text{push}_\mathcal{L}^{-1}(\vec{a}) = \partial \{\vec{p} \in \mathbb{R}^n : \vec{p} \leq \vec{a}\}$$
Hence for any $\vec{p}\in \mathbb{R}^n$ at least one coordinate of $\vec{p}$ is preserved by the push function i.e. there is some $i\in[n]$ such that $\text{push}_\mathcal{L}(\vec{p})_i= p_i$.
The push function $\text{push}_\mathcal{L}$ collapses $\mathbb{R}^n$ to the line $\mathcal{L}$ whilst preserving the partial order. Thus given a resolution of a multiparameter module $M$ the push function induces a resolution of $M^\mathcal{L}$.
\begin{lem}[Induced 1-Parameter Resolution]

Let $\mathcal{L}: \mathbb{R} \to \mathbb{R}^n$ be a positively sloped line and suppose $M\in \multimodules$ has a free resolution $F_\bullet \to M$. The $1$-parameter submodule $M^\mathcal{L} : = M\circ \mathcal{L} \in \vec{Vect}^\vec{R}$ has a corresponding free resolution $F_\bullet^\mathcal{L}$, where the generators of $F_i$ are in bijection with the generators of $F_i^\mathcal{L}$; in particular a generator $\vec{a}$ of $F_i$ corresponds to a generator $\vec{a}^\mathcal{L}$ of $F_i^\mathcal{L}$ with $\mathbb{R}$-grading given by $\mathcal{L}^{-1}\circ \text{push}_\mathcal{L}(\vec{a})$.
\end{lem}

\begin{proof}
It is straightforward to check that the pullback functor $\mathcal{L}^\ast : \multimodules \to \vec{Vect}^\vec{R}$ preserves free modules, and preserves exactness since exactness of a sequence of persistence modules may be checked pointwise.
\end{proof}

Let $M\in \vec{vect}^\vec{R}$ with free resolution $F_\bullet \to M$. The decomposition theorem of \cite{crawley-boevey_decomposition_2014}, gives rise to an isomorphism $\phi : M \to \bigoplus_{j \in \mathcal{J}}\mathds{1}^{I_j}$ for some indexed set of intervals $\{I_j = [b_j,d_j)\}_{j \in \mathcal{J}}$. For the purposes of this paper we shall require that $|J| = |\xi_0(F_0)|$ and hence allow for $I_j$ to be empty intervals. Let $F'_\bullet$ denote the free resolution of $\bigoplus_{j \in \mathcal{J}}\mathds{1}^{I_j}$, where $F'_0 = \text{Free}[\{b_j\}_{j \in \mathcal{J}}]$, $F'_1 = \text{Free}[\{d_j \neq \infty \}_{j \in \mathcal{J}}]$ and $d_j \mapsto x^{d_j-b_j}\cdot b_j$. We shall refer to $F'_\bullet$ as the \textcolor{Maroon}{canonical free resolution} of $\bigoplus_{j \in \mathcal{J}}\mathds{1}^{I_j}$.

Again let $\langle \cdot , \cdot \rangle$ denote the point-wise vector space inner product of a free module for which makes the generating set an orthonormal basis at each grade. Note that the isomorphism $\phi$ gives rise to a chain map $\phi_\bullet : F_\bullet \to F'_\bullet$.

\begin{defn}[Bars generated by $\vec{b}$ and killed by $\vec{r}$]

 For a fixed isomorphism $\phi$ and resolution $F_\bullet$ we say that (a generator of $F_0$)  \textcolor{Maroon}{${b}$ generates the bar $I_j$}  if $\text{gr}(b) = \text{gr}(b_j)$ and $\langle \phi_0({b}), b_j \rangle \neq 0 $. Moreover we say that (a generator of $F_1$) \textcolor{Maroon}{${r}$  kills the bar $I_j$} if $\text{gr}(r) = \text{gr}(r_j)$ and $\langle \phi_1({r}), {r_j} \rangle \neq 0 $.

For $M\in \fpmultimodules$ with free resolution $F_\bullet \to M$ recall there is a corresponding free resolution $F^\mathcal{L}_\bullet$ of $M^\mathcal{L}$. Given fixed decomposition isomorphisms $\phi_\mathcal{L} : M^\mathcal{L} \to \bigoplus_{j \in \mathcal{J}_\mathcal{L}}\mathds{1}^{I_j}$ for all positively sloped lines $\mathcal{L}$, we say that $\vec{b}$ (a generator of $F_0$) generates the bar $I_j$ in $M^\mathcal{L}$ if $\vec{b}^\mathcal{L}$ generates $I_j$ in $M^\mathcal{L}$. Similarly we say that the relation $\vec{r}$ (a generator of $F_1$) kills the bar $I_j$, if $\vec{r}^\mathcal{L}$ kills the bar $I_j$ in $M^\mathcal{L}$.

\end{defn}

\begin{Remark}
We can fix isomorphisms $\phi_\mathcal{L}$ such that each bar in $M^\mathcal{L}$ is generated by a unique ${b}\in \xi_0(F_0^\mathcal{L})$ and killed by a unique ${r}\in \xi_0(F_1^\mathcal{L})$. Moreover we can ensure every generator ${b}\in \xi_0(F_0^\mathcal{L})$ generates some bar in $M^\mathcal{L}$
(albeit these bars may be of zero length). See for example the reduction algorithm of \cite{edelsbrunner_computational_2010}. We shall fix our decomposition isomorphisms to have this property.
\end{Remark}

Our next lemma identifies the fibre of the push map.

\begin{lem}[Fibre of the Push Map]
Suppose $\vec{a}=(a_1,...,a_n)$, $\vec{b}=(b_1,...,b_n)$ are such that $\|\text{push}_{\mathcal{L}_\vec{b}}(\vec{a})-\vec{b}\|_\infty \leq \delta$ then $\min_{i \in [n]} | a_i - b_i| \leq \delta$. In particular, if $\vec{a} \in \Ima \mathcal{G}$ then $\vec{b}$ is $\delta$-close to $\text{Grid}_\mathcal{G}$.
 \label{lem:CloseToGrid}
\end{lem}

\begin{proof}
At least one of the coordinates $a_i$ is preserved by the push map.
\end{proof}

Suppose modules $M$ and $N$ satisfy the assumptions of \Cref{thm:LocalEquivalence}. \Cref{lem:CloseToGrid} will be used to identify that the multiparameter Betti numbers inducing births of long bars in the fibered bar code of $N$ must lie close to the grid of $M$.

In \Cref{lem:Trichotomy} we use the positions of births of bars in the fibered bar code of $M$ to limit the positions of generators of $N$. \Cref{lem:nontrivgen} establishes that every non trivial generator of $N$ generates some non-empty bar in the fibered bar code of $N$.

\begin{lem}[Non-trivial Generators Generate Non-zero Bars]
\label{lem:nontrivgen}
Let $M \in \fpmultimodules$ be a finitely presented multiparameter persistence module. If $\vec{b} \in \xi_0(M)$ then $\vec{b}$ generates a non-trivial bar in $M^{\mathcal{L}_\vec{b}}$.
\end{lem}

\begin{proof}
Suppose that $M$ has minimal free resolution $(F_\bullet,p_\bullet)$, decomposition isomorphism $\phi : M^{\mathcal{L}_\vec{b}} \Tilde{\to} \bigoplus_{j \in \mathcal{J}_\mathcal{L}}\mathds{1}^{I_j}$, and $\bigoplus_{j \in \mathcal{J}_\mathcal{L}}\mathds{1}^{I_j}$ has canonical free resolution $(F_\bullet',p_\bullet')$. Suppose that $\phi_0(\vec{b}^{\mathcal{L}_\vec{b}}) = b_j$ and $I_j = \emptyset$. By the exactness of the resolution $(F_\bullet^{\mathcal{L}_\vec{b}}, p_\bullet^{\mathcal{L}_\vec{b}})$ there is some $\vec{r}\in \xi_1(M)$ such that $\phi_1(\vec{r}^{\mathcal{L}_\vec{b}}) = r_j$ and $\text{gr}(\vec{b}^{\mathcal{L}_\vec{b}}) = \text{gr}(\vec{r}^{\mathcal{L}_\vec{b}})$. The fact that $\phi_1(\vec{r}^{\mathcal{L}_\vec{b}}) = r_j$ implies that $\langle p_1(\vec{r}), \vec{b} \rangle \neq 0$. Moreover $\langle p_1(\vec{r}), \vec{b} \rangle \neq 0$ implies that $\text{gr}(\vec{r})\geq \text{gr}(\vec{b})$. Together $\text{gr}(\vec{r})\geq \text{gr}(\vec{b})$ and $\text{gr}(\vec{b}^{\mathcal{L}_\vec{b}}) = \text{gr}(\vec{r}^{\mathcal{L}_\vec{b}})$ imply that $\text{gr}(\vec{r}) = \text{gr}(\vec{b})$. Thus $\vec{b}$ and $\vec{r}$ satisfy the hypotheses of \Cref{lem:Cancellation} and are a cancellation pair, contradicting the minimality of $(F_\bullet,p_\bullet)$.


\end{proof}

Similarly, in \Cref{lem:DichotomyRelation} we use the positions of deaths of bars in the fibered bar code of $M$ to determine the positions of relations of $N$. \Cref{lem:nontrivrel} establishes that every non trivial relation of $N$ kill some non-empty bar in the fibered bar code of $N$.

\begin{lem}[Non-trivial Relations Kill]
\label{lem:nontrivrel}
Let $M \in \fpmultimodules$ be a finitely presented multiparameter persistence module. If $\vec{r} \in \xi_1(M)$ then $\vec{r}^{\mathcal{L}_\vec{r}} \notin \ker p_1^{\mathcal{L}_\vec{r}}$ 
and so $\vec{r}$ kills a bar in $M^{\mathcal{L}_\vec{r}}\cong_\phi \bigoplus_{j \in \mathcal{J}_\mathcal{L}}\mathds{1}^{I_j}$. 
\end{lem}

\begin{proof}
Suppose that $M$ has minimal free resolution $(F_\bullet,p_\bullet)$, decomposition isomorphism $\phi : M^{\mathcal{L}_\vec{r}} \Tilde{\to} \bigoplus_{j \in \mathcal{J}_\mathcal{L}}\mathds{1}^{I_j}$, and $\bigoplus_{j \in \mathcal{J}_\mathcal{L}}\mathds{1}^{I_j}$ has canonical free resolution $(F_\bullet',p_\bullet')$.
If $\vec{r}^{\mathcal{L}_\vec{r}} \in \ker p_1^{\mathcal{L}_\vec{r}}$ then by exactness of the resolution $(F_\bullet^{\mathcal{L}_\vec{r}}, p_\bullet^{\mathcal{L}_\vec{r}})$  there is some $\vec{s}\in \xi_2(M)$ such that $\langle p_2^{\mathcal{L}_\vec{r}}(\vec{s}^{\mathcal{L}_\vec{r}}), \vec{r}^{\mathcal{L}_\vec{r}} \rangle \neq 0$ and $\text{gr}(\vec{r}^{\mathcal{L}_\vec{r}}) = \text{gr}(\vec{s}^{\mathcal{L}_\vec{r}})$. The fact that $\langle p_2^{\mathcal{L}_\vec{r}}(\vec{s}^{\mathcal{L}_\vec{r}}), \vec{r}^{\mathcal{L}_\vec{r}} \rangle \neq 0$ implies that $\text{gr}(\vec{s})\geq \text{gr}(\vec{r})$. Together $\text{gr}(\vec{s})\geq \text{gr}(\vec{r})$ and $\text{gr}(\vec{r}^{\mathcal{L}_\vec{r}}) = \text{gr}(\vec{s}^{\mathcal{L}_\vec{r}})$ imply that $\text{gr}(\vec{s}) = \text{gr}(\vec{r})$. Thus $\vec{r}$ and $\vec{s}$ satisfy the hypotheses of \Cref{lem:Cancellation} and are a cancellation pair, contradicting the minimality of $(F_\bullet,p_\bullet)$.
Suppose that $\vec{r}$ does not kill a bar and so $\phi_1(\vec{r}^{\mathcal{L}_\vec{r}}) = 0$. Since $(F_\bullet',p_\bullet')$ is the canonical free resolution $\phi_0$ is an isomorphism, and so $\phi_1(\vec{r}^{\mathcal{L}_\vec{r}}) = 0$ implies that $\vec{r}^{\mathcal{L}_\vec{r}}\in\ker p_1$. Thus it must be that $\vec{r}$ kills a bar in $M^{\mathcal{L}_\vec{r}}\cong_\phi \bigoplus_{j \in \mathcal{J}_\mathcal{L}}\mathds{1}^{I_j}$.

\end{proof}

Through the proof of \Cref{lem:nontrivrel}, we observe that a relation $\vec{r}\in \xi_1(M)$which does not kill a bar must instead be lie in the image of a second order relation $\vec{s}\in \xi_2(M)$.

\begin{defn}[Unmerge Function]

Let $\mathcal{G}= \Pi \mathcal{G}^i$ be a grid with $\delta < \frac{c(\mathcal{G})}{2}$, and let $ \vec{e}_i$ denote the standard basis vector of $\mathbb{R}^n$. The \textcolor{Maroon}{unmerge function} $\Unmerge$ maps points lying on $\text{Grid}_\mathcal{G}$ to lie $\delta$ far away from the grid: $\Unmerge: \text{Grid}_\mathcal{G} \to \mathbb{R}^n$. For $\vec{a} \in \text{Grid}_\mathcal{G}$ let $\mathcal{I}_\vec{a}$ be the indices of the coordinates of $\vec{a}$ such that $|a_i - \Ima \mathcal{G}^i| \leq \delta$. We define the $\delta$-unmerge of $\vec{a}$ as follows:
$$ \Unmerge(\vec{a}) = \Merge(\vec{a}) + \delta \sum_{i \in \mathcal{I}_\vec{a}} \vec{e}_i$$

\end{defn}

By construction the unmerge function takes $\vec{a}$ and returns the maximum element in the poset $\mathbb{R}^n$ which merges to same point as $\vec{a}$ that is: 
$$\Unmerge(\vec{a}) = \max \{\vec{p} \in \mathbb{R}^n : \Merge(\vec{p}) = \Merge(\vec{a})\}$$
Thus for all $\vec{a} \in \text{Grid}_\mathcal{G}$ we have $\|\Unmerge(\vec{a}) - \text{Grid}_\mathcal{G} \|_\infty= \delta$. See \Cref{fig:Unmerge} for examples of the unmerge function.

\begin{figure}
    \centering
    \begin{subfigure}{0.45\textwidth}
        \centering
        \includegraphics[width=1\textwidth]{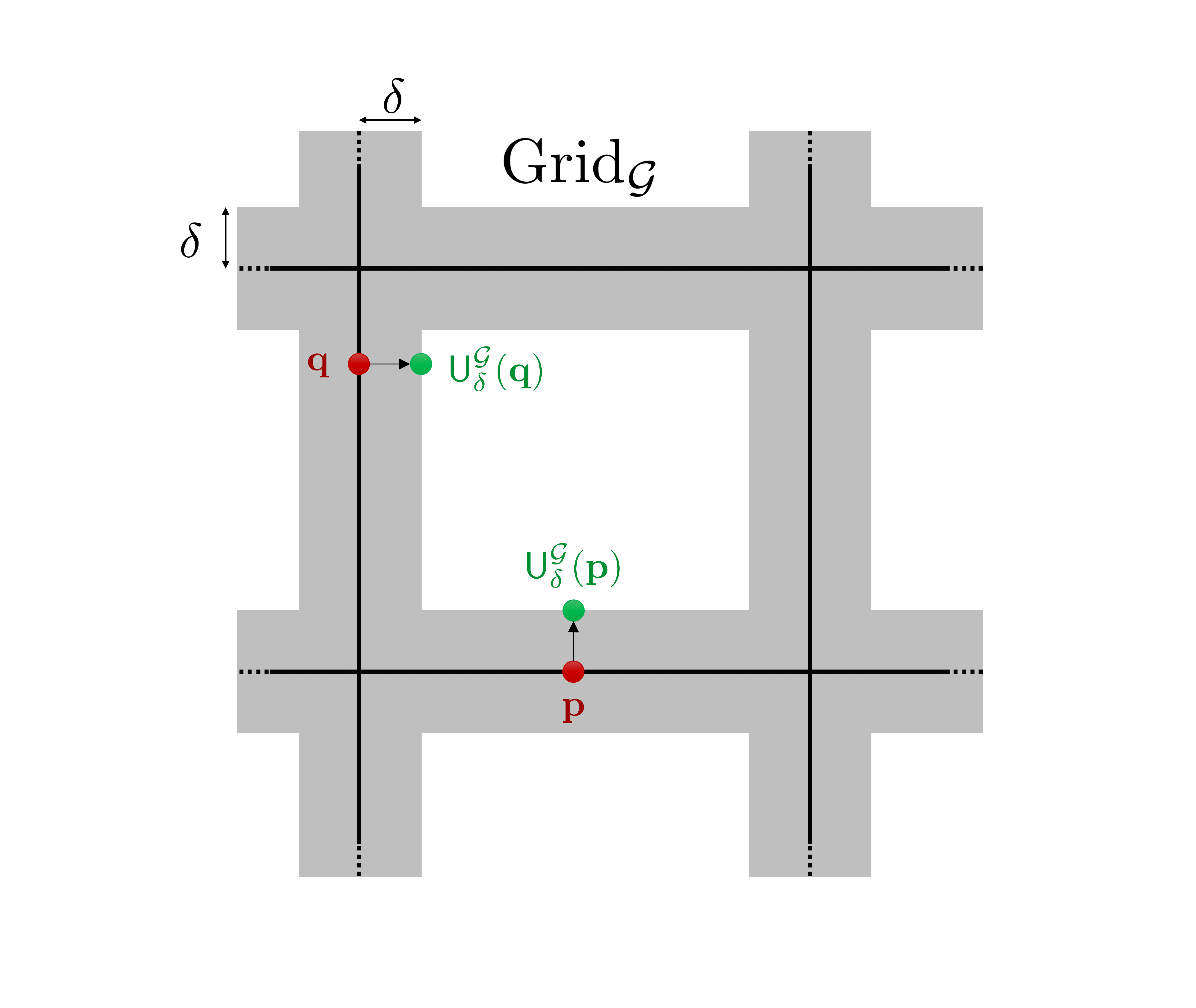} 
        \caption{$\Unmerge$ applied to points $\vec{p},\vec{q}\in \mathbb{R}^2$.}
    \end{subfigure}\hfill
    \begin{subfigure}{0.45\textwidth}
        \centering
        \includegraphics[width=1\textwidth]{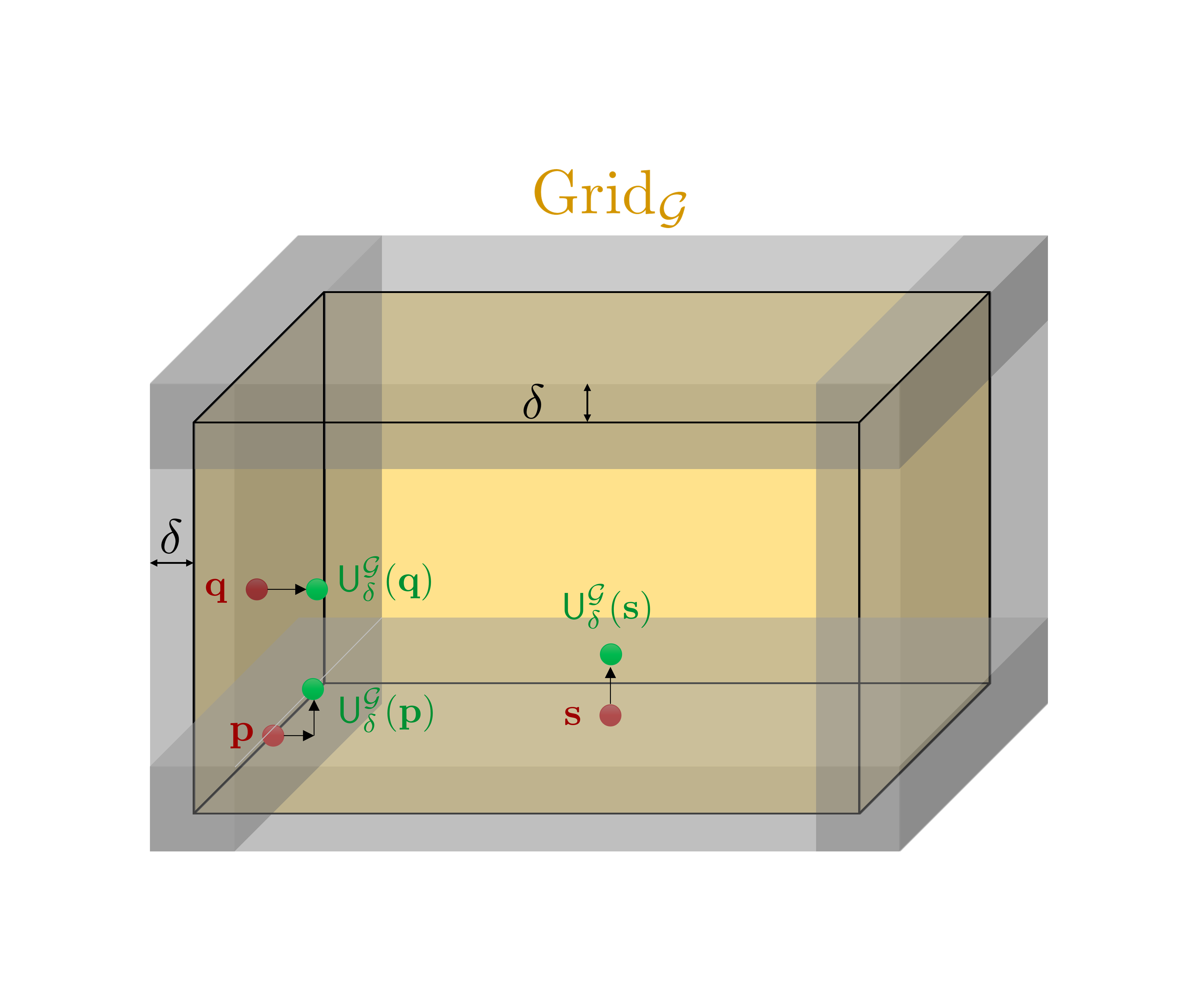} 
        \caption{$\Unmerge$ applied to points $\vec{p},\vec{q},\vec{s}\in \mathbb{R}^3$.}
    \end{subfigure}
    \caption{Examples of the unmerge function $\Unmerge$ moving points from $\text{Grid}_\mathcal{G}$ to lie $\delta$-far away from $\text{Grid}_\mathcal{G}$.}
    \label{fig:Unmerge}
\end{figure}

The following result, \Cref{lem:GradingMerge}, will be used to show that the grades of multiparameter Betti numbers of $N$ merge to the same grade and satisfy the hypotheses of \Cref{lem:Cancellation} and thus cancel each other out.

\begin{lem}[Sufficient Conditions for Grades to Merge]
Let $\vec{a},\vec{b}\in (\mathbf{R}^n,\leq)$ and $\mathcal{G}$ be a grid function such that:
$\|\vec{a} -\text{Grid}_\mathcal{G}\|_\infty \leq \delta$
and $\vec{b}\geq\vec{a}$.
If $\text{push}_{\mathcal{L}_{\Unmerge\Merge({\vec{a}})}}({\vec{a}})= \text{push}_{\mathcal{L}_{\Unmerge\Merge({\vec{a}})}}({\vec{b}})$, then $\Merge({\vec{a}}) =\Merge({\vec{b}})$.
\label{lem:GradingMerge}
\end{lem}

\begin{proof}
Since $\text{push}_{\mathcal{L}_{\Unmerge\Merge({\vec{a}})}}({\vec{a}})= \text{push}_{\mathcal{L}_{\Unmerge\Merge({\vec{a}})}}({\vec{b}})$ and $\vec{b}\geq\vec{a}$, then $\vec{b} \in \text{push}^{-1}_{\mathcal{L}_{\Unmerge\Merge({\vec{a}})}}(0) \cap \{\geq \vec{a}\} = [\vec{a},\Unmerge\Merge({\vec{a}})]$, where $[\vec{a},\Unmerge\Merge({\vec{a}})]$ denotes the interval between $\vec{a}$ and $\Unmerge\Merge({\vec{a}})$. By definition of the unmerge function, the coordinate values of $\vec{a}$ which are unchanged by $\Merge$ are also unchanged when applying $\Unmerge$ to $\Merge(\vec{a})$. Thus $\vec{a}$ and $\vec{b}$ share coordinate values for all the coordinate values of $\vec{a}$ which are unchanged by $\Merge$. Since $\|\Unmerge\Merge(\vec{a})- \text{Grid}_\mathcal{G}\|_\infty = \delta$, then $\|\vec{b}- \text{Grid}_\mathcal{G}\|_\infty \leq \delta$, and so the remaining coordinates of $\vec{b}$ are merged to the same coordinate values of $\Merge(\vec{a})$, thus it follows that $\Merge({\vec{a}}) =\Merge({\vec{b}})$.
\end{proof}

\Cref{lem:Trichotomy} and \Cref{lem:DichotomyRelation} use the small matching distance to derive properties about the gradings of the generators and relations of $N$ and the length of the bars the generators and relations birth and kill in the fibered bar code.

\begin{lem}[Properties of Generators]

Suppose $d_0(M,N) < \kapeps$ and $c_M > 2 \kapeps$. For all $\vec{b}\in \xi_0(N)$, at least one of the following properties holds:
\begin{enumerate}[label=\text{B.\arabic*}]
    \item \label{ppty:CloseToGrid} $\vec{b}$ lies $2\kapeps$-close to $\Ima \mathcal{G}_M$ 
    \item \label{ppty:SmallNearb} $\vec{b}$ lies $\kapeps$-close to $\Grid$ and generates a bar of length less than $2\kapeps$ in the module $N^{\mathcal{L}_{\unmerge{\kappa\epsilon}{\mathcal{G}_M}\textsf{M}^{\mathcal{G}_M}_{\kapeps}(\vec{b})}}$
    \item \label{ppty:SmallInb}$\vec{b}$ generates a bar of length less than $2\kapeps$ in the module $N^{\mathcal{L}_\vec{b}}$ 
\end{enumerate}
\label{lem:Trichotomy}
\end{lem}

\begin{proof}
Consider the module $N^{\mathcal{L}_\vec{b}}$. Since $\vec{b}\in \xi_0(N)$ is a non-trivial generator, $\vec{b}$ generates a non-trivial bar in $N^{\mathcal{L}_\vec{b}}$ (\Cref{lem:nontrivgen}). Suppose the bar in $N^{\mathcal{L}_\vec{b}}$ generated by $\vec{b}$ has length  $< 2 \kapeps$, then we satisfy property \ref{ppty:SmallInb}. 

Else the bar generated by $\vec{b}$ in $N^{\mathcal{L}_\vec{b}}$ has length $\geq 2 \kapeps$ and must be $\kapeps$-matched with a bar in the module $M^{\mathcal{L}_\vec{b}}$. Hence by \Cref{lem:CloseToGrid}, $\vec{b}$ is $\kapeps$-close to $\Grid$. If $\vec{b}$ is $2\kapeps$-close to $\Ima \mathcal{G}_M$ we satisfy \ref{ppty:CloseToGrid}, else $\vec{b}$ is $> 2 \kapeps$ from all grid points of $M$.

$\textsf{M}_{\kapeps}^{\mathcal{G}_M}(\vec{b})$ lies on $\Grid$ but is not in the image of $\mathcal{G}_M$. 
Using the fact that $c_M> 2\kapeps$ and at least one of the coordinates of $\vec{b}$ is unchanged by $\textsf{M}_{\kapeps}^{\mathcal{G}_M}$, we have that the $\kapeps$-close line $\mathcal{L}_{\unmerge{\kapeps}{\mathcal{G}_M}\textsf{M}_{\kapeps}^{\mathcal{G}_M}(\vec{b})}$ is such that $\|\text{push}_{\mathcal{L}_{\unmerge{\kapeps}{\mathcal{G}_M}\textsf{M}_{\kapeps}^{\mathcal{G}_M}(\vec{b})}}(\vec{b})-\text{push}_{\mathcal{L}_{\unmerge{\kapeps}{\mathcal{G}_M}\textsf{M}_{\kapeps}^{\mathcal{G}_M}(\vec{b})}}(\mathcal{G}_M)\|_\infty=  \kapeps$.
Since $\text{push}_{\mathcal{L}_{\unmerge{\kapeps}{\mathcal{G}_M}\textsf{M}_{\kapeps}^{\mathcal{G}_M}(\vec{b})}}(\textsf{M}_{\kapeps}^{\mathcal{G}_M}(\vec{b}))$ is $\kapeps$ far from all points of $\text{push}_{\mathcal{L}_{\unmerge{\kapeps}{\mathcal{G}_M}\textsf{M}_{\kapeps}^{\mathcal{G}_M}(\vec{b})}}(\mathcal{G}_M)$ and $d_0(M,N)< \kapeps$, then the bar generated by $\vec{b}$ in ${N}^{\mathcal{L}_{\unmerge{\kapeps}{\mathcal{G}_M}\textsf{M}_{\kapeps}^{\mathcal{G}_M}(\vec{b})}}$ cannot be matched to bars in $M^{\mathcal{L}_{\unmerge{\kapeps}{\mathcal{G}_M}\textsf{M}_{\kapeps}^{\mathcal{G}_M}(\vec{b})}}$ and so must be length $<2\kappa \epsilon$.
\end{proof}


\begin{lem}[Properties of Relations]
Suppose $d_0(M,N) < \kapeps$, $\xi_0(N)\subset \Ima \mathcal{G}_M$  and that $c_M > 4 \kappa \epsilon$. For all $\vec{r}\in \xi_1(N)$, one of the following properties holds:
\begin{enumerate}[label=\text{R.\arabic*}]
\item $\vec{r}$ is $2\kapeps$-close to $\Ima \mathcal{G}_M$
\item $\vec{r}$ is $2\kapeps$-close to $\Grid$ and for all $\alpha \in (0,\frac{c_M - 4\kapeps}{2})$ if it kills a bar in the module $N^{\mathcal{L}_{\unmerge{2\kapeps + \alpha}{\mathcal{G}_M}\textsf{M}_{2\kapeps}^{\mathcal{G}_M}(\vec{r})}}$ that bar is of length less than $2\kapeps$
\end{enumerate}
\label{lem:DichotomyRelation}
\end{lem}

\begin{proof}
For each ${\vec{r}} \in \xi_1({N})$ 
consider the bar killed by ${\vec{r}}$ in the module 
${N}^{\mathcal{L}_{{\vec{r}}}}$ (\Cref{lem:nontrivrel}). Since 
$d_I(M,{N})< \kappa \epsilon$ the bar killed by 
${\vec{r}}$ is either $ \kapeps$-matched to a bar of 
$M^{\mathcal{L}_{{\vec{r}}}}$ and thus $\vec{r}$ is $\kapeps$-close to $\Grid$ (by \Cref{lem:CloseToGrid}), else the bar is $2 \kapeps$-trivial and 
unmatched -- in which case it is $2\kapeps$-close to a coordinate of a generator of $N$ and since $\xi_0(N)\subset  \Ima \mathcal{G}_M $ then $\vec{r}$ is $2\kapeps$-close to $\Grid$. In either case
${\vec{r}}$ is $2 \kapeps$-close to $\Grid$.

If $\|{\vec{r}}- \Ima \mathcal{G}_M\|_\infty > 2 \kapeps$ then at least one of the coordinates of $\vec{r}$ is unchanged by $\textsf{M}_{2\kapeps}^{\mathcal{G}_M}$ and so we yield that: 
$$\|\text{push}_{\mathcal{L}_{\unmerge{2\kapeps+\alpha}{\mathcal{G}_M}\textsf{M}_{2\kapeps}^{\mathcal{G}_M}({\vec{r}})}}(\vec{r})-\text{push}_{\mathcal{L}_{\unmerge{2\kapeps+\alpha}{\mathcal{G}_M}\textsf{M}_{2\kapeps}^{\mathcal{G}_M}({\vec{r}})}}( \Ima \mathcal{G}_M)\|_\infty =  2\kapeps+\alpha$$ 
Hence any bar killed by $\vec{r}$ in the module $N^{\mathcal{L}_{\unmerge{2\kapeps}{\mathcal{G}_M}\textsf{M}_{2\kapeps}^{\mathcal{G}_M}(\vec{r})}}$ is unmatched and so of length less than $2\kapeps$ as $d_0(M,N)<\kapeps$.
\end{proof}


Using information about the location of the generators and relations of $N$ together with which generators and relations birth or kill small bars, we find a nearby module $\Tilde{N}$ respecting the grid of $M$. The nearby module is attained by a series of merges and simplifications applied to $N$. 

\begin{lem}[$\cup_i \xi_i(\Tilde{N}) \subset\Ima \mathcal{G}$]
Suppose $d_0(M,N) < \kapeps$, $d_I(M,N) = \varepsilon$ and $c_M > 40\kappa \epsilon$. There exists a module $\Tilde{N}$ such that $\cup_i \xi_i(\Tilde{N}) \subset \Ima \mathcal{G}_M$, and $d_I(N,\Tilde{N})\leq 34 \kappa \epsilon$.
\label{lem:UniteGrids}
\end{lem}

\begin{proof}
\Cref{lem:Trichotomy} establishes that for all $\vec{b}\in \xi_0(N)$ at least one of the following holds:
\begin{enumerate}[label=\text{B.\arabic*}]
    \item  $\vec{b}$ lies $2\kapeps$-close to $\Ima \mathcal{G}_M$ 
    \item  $\vec{b}$ lies $\kapeps$-close to $\Grid$ and generates a bar of length less than $2\kapeps$ in the module $N^{\mathcal{L}_{\unmerge{\kappa\epsilon}{\mathcal{G}_M}\textsf{M}^{\mathcal{G}_M}_{\kapeps}(\vec{b})}}$
    \item \label{item:property3}$\vec{b}$ generates a bar of length less than $2\kapeps$ in the module $N^{\mathcal{L}_\vec{b}}$ 

\end{enumerate}

Consider $N' = \textsf{S}_{2\kapeps}(N)$. By \Cref{lem:BarcodeSimp} every generator $\vec{b}$ satisfying \ref{item:property3} now generates a bar of zero length in the module $N^{\mathcal{L}_\vec{b}}$.
\Cref{lem:nontrivgen} implies that such a generator satisfying \ref{item:property3} is no longer a non-trivial generator of $N'$.
Moreover by \Cref{lem:SimplificationGenerators} the grades of generators $\vec{b}'\in \xi_0(N')$ are inherited from generators of some $\vec{b} \in \xi_0(N)$.
Thus for every $\vec{b}'\in \xi_0(N')$ at least one of the following holds:

\begin{enumerate}[label=\text{B'.\arabic*}]
    \item \label{item:property'1} $\vec{b}'$ lies $2\kapeps$-close to $\Ima\mathcal{G}_M$  
    \item \label{item:property'2} $\vec{b}'$ lies $\kapeps$-close to $\Grid$ and generates a bar of zero-length in the module $N'^{\mathcal{L}_{\unmerge{\kappa\epsilon}{\mathcal{G}_M}\textsf{M}^{\mathcal{G}_M}_{\kapeps}(\vec{b})}}$
\end{enumerate}

Consider $\hat{N} = \textsf{M}^{\mathcal{G}_M}_{2\kapeps}(N')$. Every generator $\vec{b}'$ satisfying \ref{item:property'1} now lies on the grid $\mathcal{G}_M$. Every generator $\vec{b}'$ satisfying \ref{item:property'2} lies $\kappa \epsilon$-close to $\Grid$ and there is a relation $\vec{r}'$ which kills the zero bar generated by $\vec{b}'$ in ${N'}^{\mathcal{L}_{\unmerge{\kappa\epsilon}{\mathcal{G}_M}\textsf{M}^{\mathcal{G}_M}_{\kapeps}(\vec{b}')}}$. \Cref{lem:GradingMerge} implies that under $\textsf{M}^{\mathcal{G}_M}_{2\kapeps}$ the generator $\vec{b}'$ and relation $\vec{r}'$ pair are brought to the same grade. 
Hence $\textsf{M}^{\mathcal{G}_M}_{2\kapeps}(\vec{b}')$ and $\textsf{M}^{\mathcal{G}_M}_{2\kapeps}(\vec{r}')$ satisfy the hypotheses of \Cref{lem:Cancellation}, and so any generator which satisfied \ref{item:property'2} is no longer a non-trivial generator of $\hat{N}$.
Hence all non-trivial generators lie on the grid,  $\xi_0(\hat{N})\subset \Ima\mathcal{G}_M$, and by construction $d_I(M,\hat{N})< \kapeps + 2\kapeps + 2\kapeps \leq 5 \kapeps$.


For each $\hat{\vec{r}} \in \xi_1(\hat{N})$ the hypotheses of \Cref{lem:DichotomyRelation} hold (replacing $\kapeps$ with $5\kappa \epsilon$), thus $\hat{\vec{r}}$ is $10 \kapeps$-close to $\Grid$. More precisely for each relation $\hat{\vec{r}} \in \xi_1(\hat{N})$ at least one of the following holds:

\begin{enumerate}[label=\text{R.\arabic*}]
    \item \label{item:relproperty1} $\hat{\vec{r}}$ lies $10\kapeps$-close to $\Ima \mathcal{G}_M$ 
    \item \label{item:relproperty2}$\hat{\vec{r}}$ lies $10 \kapeps$-close to $\Grid$ and for all $\alpha \in (0,\frac{c_M - 10\kapeps}{2})$ if it kills a bar in the module $\hat{N}^{\mathcal{L}_{\unmerge{10\kapeps +\alpha}{\mathcal{G}_M}\textsf{M}_{10\kapeps}^{\mathcal{G}_M}(\vec{r})}}$ that bar is of length less than $10\kapeps$ 
\end{enumerate}

Consider $\Tilde{N} = \textsf{M}^{\mathcal{G}_M}_{20\kapeps}\circ \textsf{S}_{10\kapeps}(\hat{N})$. Any relation $\hat{\vec{r}} \in \xi_1(\hat{N})$ satisfying \ref{item:relproperty1} is shifted in grading by no more than $10\kapeps$ by the functor $\textsf{S}_{10\kapeps}$ and thus lies on a grid point of $\mathcal{G}_M$ after applying $\textsf{M}^{\mathcal{G}_M}_{20\kapeps}$.

Consider instead a relation $\hat{\vec{r}} \in \xi_1(\hat{N})$ satisfying \ref{item:relproperty2} and not \ref{item:relproperty1}. Property \ref{item:relproperty2} implies that $\textsf{M}^{\mathcal{G}_M}_{20\kapeps}(\hat{\vec{r}})$ lies on $\Grid$. 
Moreover property \ref{item:relproperty2} implies that if $\hat{\vec{r}}$ killed a bar in the module $\hat{N}^{\mathcal{L}_{\unmerge{10\kapeps+\alpha}{\mathcal{G}_M}\textsf{M}_{10\kapeps}^{\mathcal{G}_M}(\vec{r})}}$ it was of length less than $10\kapeps$. The effect of the simplification functor on bar codes (\Cref{lem:BarcodeSimp}) together with our previous observation implies that if $\textsf{S}_{10\kapeps}(\hat{\vec{r}})$ kills a bar in the module 
$\textsf{S}_{10\kapeps}({\hat{N})^{\unmerge{10\kapeps+\alpha}{\mathcal{G}_M}\textsf{M}_{10\kapeps}^{\mathcal{G}_M}(\vec{r})}}$ then that bar is of length zero.

It is not possible for $\textsf{S}_{10\kapeps}(\hat{\vec{r}})$ to kill a zero length bar in $\textsf{S}_{10\kapeps}({\hat{N})^{\mathcal{L}_{\unmerge{10\kapeps+\alpha}{\mathcal{G}_M}\textsf{M}_{10\kapeps}^{\mathcal{G}_M}(\vec{r})}}}$ since $\hat{\vec{r}}$ is shifted in grading by no more than $10\kapeps$ by the functor $\textsf{S}_{10\kapeps}$, and $\hat{\vec{r}}$ was $10\kapeps +\alpha$ far from $\Ima\mathcal{G}_M$ when pushed to the line $\mathcal{L}' = \mathcal{L}_{\unmerge{10\kapeps+\alpha}{\mathcal{G}_M}\textsf{M}_{10\kapeps}^{\mathcal{G}_M}(\vec{r})}$:
$$\|\text{push}_{\mathcal{L}'}(\textsf{S}_{10\kapeps}(\hat{\vec{r}}))- \text{push}_{\mathcal{L}'}(\Ima \mathcal{G}_M)\| \geq \|\text{push}_{\mathcal{L}'}(\hat{\vec{r}})- \text{push}_{\mathcal{L}'}(\Ima \mathcal{G}_M)\| - 10\kapeps = \alpha > 0$$
Hence $\hat{\vec{r}}$ cannot be pushed to the same value as a generator in the line the line $\mathcal{L}'$ and so does not kill a zero length bar.

Thus it must be that there is a second order relation $\vec{a} \in \xi_2(\textsf{S}_{10\kapeps}(\hat{N}))$ with $\text{push}_{\mathcal{L}'}(\vec{a}) = \text{push}_{\mathcal{L}'}(\textsf{S}_{10\kapeps}(\hat{\vec{r}}))$ and $\langle p_2(\vec{a}), \textsf{S}_{10\kapeps}(\hat{\vec{r}})\rangle \neq 0$.
\Cref{lem:GradingMerge} implies that under the merge $\textsf{M}^{\mathcal{G}_M}_{20\kapeps}$ the relation $\textsf{S}_{10\kapeps}(\hat{\vec{r}})$ and second order relation $\vec{a}$ are brought to the same grade, $\text{gr}(\textsf{M}^{\mathcal{G}_M}_{20\kapeps}\textsf{S}_{10\kapeps}(\hat{\vec{r}}))=\text{gr}(\textsf{M}^{\mathcal{G}_M}_{20\kapeps}({\vec{a}}))$. Under this merge the hypotheses of \Cref{lem:Cancellation} are satisfied for the relation $\textsf{M}^{\mathcal{G}_M}_{20\kapeps}\textsf{S}_{10\kapeps}(\hat{\vec{r}})$ and second-order relation $\textsf{M}^{\mathcal{G}_M}_{20\kapeps}(\vec{a})$ and thus $\textsf{M}^{\mathcal{G}_M}_{20\kapeps}\textsf{S}_{10\kapeps}(\hat{\vec{r}})$ is no longer a non-trivial relation.

In total we observe that $\Tilde{N} = \textsf{M}^{\mathcal{G}_M}_{20\kapeps}\circ \textsf{S}_{10\kapeps}\circ \textsf{M}^{\mathcal{G}_M}_{2\kapeps}\circ \textsf{S}_{2\kapeps}(N)$ is such that  $ \xi_0(\Tilde{N})\cup \xi_1(\Tilde{N})\subset \mathcal{G}_M$ and \Cref{lem:xi0xi1detgrid} implies that in fact $\cup_i \xi_i(\Tilde{N})\subset \mathcal{G}_M$. Since the simplification functor $\textsf{S}_{\delta}$ and merge functor $\Merge$ each perturb a module by no more that $\delta$ in the interleaving distance we have that $d_I(N,\Tilde{N})\leq 20\kapeps+10\kapeps+2\kapeps+2\kapeps= 34\kapeps$.




 
\end{proof}

\begin{Remark}
Whilst not explicitly remarked upon in the proof of \Cref{lem:UniteGrids}, one should note that the cancellations of multiparameter Betti numbers caused by bringing these Betti numbers to the same grade by applying a merge functor hold up to multiplicity. For example, if there are $l$ generators $\vec{b}_1,...,\vec{b}_l$ satisfying \ref{ppty:SmallNearb} which all merge to the same grade, then $\unmerge{\kappa\epsilon}{\mathcal{G}_M}\textsf{M}^{\mathcal{G}_M}_{\kapeps}(\vec{b_1})=...=\unmerge{\kappa\epsilon}{\mathcal{G}_M}\textsf{M}^{\mathcal{G}_M}_{\kapeps}(\vec{b_l})$ and thus there are $l$ relations killing the small bars generated by the $\vec{b}_1,...,\vec{b}_l$ in the module $N^{\mathcal{L}_{\unmerge{\kappa\epsilon}{\mathcal{G}_M}\textsf{M}^{\mathcal{G}_M}_{\kapeps}(\vec{b_1})}}$, and so all these generators will be cancelled.
\end{Remark}











\Cref{lem:LesnickLemma} is a specialised version of {\cite[Lemma~6.7]{lesnick_theory_2015}}:

\begin{lem}[An Interleaving Smaller than the Controlling Constant Yields an Isomorphism {\cite[Lemma~6.7]{lesnick_theory_2015}}]
Suppose $M,N \in \fpmultimodules$ are $\varepsilon$-interleaved and that for all $\vec{a}\in \bigcup_i \xi_i(M) \cup \xi_i(N)$ the $2\varepsilon$ transition morphism $(M_\vec{a}\to M_{\vec{a}+2\varepsilon})$ is an isomorphism of vector spaces and the $2\varepsilon$ transition morphism $(N_\vec{a}\to N_{\vec{a}+2\varepsilon})$ is also an isomorphism of vector spaces, then $M,N$ are isomorphic.
\label{lem:LesnickLemma}
\end{lem}

Combining the results of this section we attain a short proof of our main result.

\begin{proof}\textit{(\Cref{thm:LocalEquivalence})}
Suppose (seeking contradiction) that $d_0(M,N) < \kapeps$. Then by \Cref{lem:UniteGrids} there is a module $\Tilde{N}$ with  $\cup_i \xi_i(\Tilde{N}) \subset \Ima \mathcal{G}_M$, and $d_I(N,\Tilde{N})\leq 34 \kappa \epsilon$. By the triangle inequality $d_I(N,M)\leq \varepsilon + 34 \kappa \epsilon$. Since $c_M > 2 \varepsilon + 68 \kapeps $, the internal morphisms $(\Tilde{N}_\vec{a}\to \Tilde{N}_{\vec{a}+2\varepsilon + 68\kapeps})$, $(M_\vec{a}\to M_{\vec{a}+2\varepsilon + 68\kapeps})$ are isomorphisms for all $\vec{a}\in \cup_i \xi_i(\Tilde{N}) \cup \cup_i \xi_i(M) \subset  \Ima \mathcal{G}_M$. Hence \Cref{lem:LesnickLemma} implies that $d_I(M,\Tilde{N})=0$. This leads to a contradiction for all $\kappa \in [0,\frac{1}{34})$: $$\varepsilon= d_I(M,N) \leq  d_I(M,\Tilde{N}) + d_I(\Tilde{N},N) \leq 34 \kapeps < \varepsilon = d_I(M,N).$$
\end{proof}

\section{Equivalence of Intrinsic Metrics}
\label{sec:GlobalEquiv}

In this section we show that the local equivalence result in the previous section induces a global equivalence of intrinsic metrics. A consequence of the Characterisation Theorem in {\cite[Theorem~4.4]{lesnick_theory_2015}} (which characterises the presentations of $\varepsilon$-interleaved multiparameter persistences modules), is that the space of finitely presented multiparameter persistence modules equipped with the interleaving distance is a path metric space. That is to say a pair of finitely presented modules at distance $\varepsilon$ in the interleaving distance are joined by an $\varepsilon$-length path. Consequently, we deduce a global equivalence of intrinsic metrics between the induced intrinsic matching distance $\hat{d}_0$ and the interleaving distance \Cref{thm:GlobalEquivalence}. For a more thorough introduction to path metric spaces and induced intrinsic metrics see \cite{gromov_metric_1998}. 

We begin by recalling the characterisation of interleavings proved by Lesnick:

\begin{thm}[Interleaving Characterisation {\cite[Theorem~4.4]{lesnick_theory_2015}}]
Let $M,N \in \vec{Vect}^{\vec{R}^n}$ then $M,N$ are $\varepsilon$-interleaved if and only if there exist presentations:
\begin{align*}
    M \cong \langle \mathcal{X}_M, \mathcal{X}_N(-\varepsilon) | \mathcal{R}_M, \mathcal{R}_N(-\varepsilon) \rangle \\
    N \cong \langle \mathcal{X}_M(-\varepsilon), \mathcal{X}_N | \mathcal{R}_M(-\varepsilon), \mathcal{R}_N \rangle
\end{align*}
\label{thm:Characterisation}
\end{thm}

\begin{Remark}
The presentations in the result of \Cref{thm:Characterisation} are constructed using the data of the interleaving morphisms. We note that in this construction each relation in $\mathcal{R}_M$ is a homogeneous element of the free module
$\text{Free}[\mathcal{X}_M \sqcup \mathcal{X}_N(-\varepsilon)]$. Similarly, each relation in $\mathcal{R}_N$ is a homogeneous element of the free module
$\text{Free}[\mathcal{X}_M(-\varepsilon) \sqcup \mathcal{X}_N]$.
This means that $\langle \mathcal{X}_M(-t\varepsilon), \mathcal{X}_N((t-1)\varepsilon) | \mathcal{R}_M(-t\varepsilon), \mathcal{R}_N((t-1)\varepsilon)) \rangle$ 
is a well-defined presentation for all $t\in[0,1]$.
\label{rmk:welldefined}
\end{Remark}

A straight forward consequence of the characterisation of interleavings is the following result:

\begin{cor}[$(\fpmultimodules,d_I)$ is a Path Metric Space]
The extended metric space of finitely presented multiparameter persistence modules equipped with the interleaving distance, $(\fpmultimodules,d_I)$, is a path metric space. 
\label{cor:PathMetricSpace}
\end{cor}
\begin{proof}
Given a pair of finitely presented modules $M,N$ with $d_I(M,N)= \varepsilon$, since the modules are finitely presented, the interleaving distance is realised and so they are $\varepsilon$-interleaved {\cite[Theorem~6.1]{lesnick_theory_2015}}. \Cref{thm:Characterisation} implies that there exist presentations $M \cong \langle \mathcal{X}_M, \mathcal{X}_N(-\varepsilon) | \mathcal{R}_M, \mathcal{R}_N(-\varepsilon) \rangle$, $N \cong \langle \mathcal{X}_M(-\varepsilon), \mathcal{X}_N | \mathcal{R}_M(-\varepsilon), \mathcal{R}_N \rangle$.
Consider the path $\gamma : [0,1] \to \fpmultimodules$ defined by:
$$
\gamma(t) = 
\langle \mathcal{X}_M(-t\varepsilon), \mathcal{X}_N((t-1)\varepsilon) | \mathcal{R}_M(-t\varepsilon), \mathcal{R}_N(t-1)\varepsilon) \rangle 
$$

The path $\gamma$ is well-defined by \Cref{rmk:welldefined}. Let $t<s$ and  $\mathcal{X}_t = \mathcal{X}_M(-t\varepsilon), 
\mathcal{X}_s = \mathcal{X}_N((s-1)\varepsilon), 
\mathcal{R}_t = \mathcal{R}_M(-t\varepsilon),$ and $\mathcal{R}_s = \mathcal{R}_N((s-1)\varepsilon)$.
Observe that:
$$
\gamma(t) = 
\langle \mathcal{X}_t, \mathcal{X}_s((t-s)\varepsilon) | \mathcal{R}_t, \mathcal{R}_s((t-s)\varepsilon) \rangle 
\hspace{1cm}
\gamma(s) = 
\langle \mathcal{X}_t((t-s)\varepsilon), \mathcal{X}_s | \mathcal{R}_t((t-s)\varepsilon), \mathcal{R}_s \rangle 
$$
Hence using \Cref{thm:Characterisation} we see that $d_I(\gamma(t),\gamma(s)) \leq |s-t|\varepsilon$.
\end{proof}


In a metric space $(X,d)$ one can define the length of a continuous path $\gamma$:

\begin{defn}[Length of a Path \cite{gromov_metric_1998}]
Define the \textcolor{Maroon}{length of a path} $\gamma : [0,1] \to X$  continuous with respect to the metric $d$ to be given by:

$$\ell_d(\gamma) = \sup_{\substack{0=t_0 < t_1 < ...< t_n = 1 \\ n\in \mathbb{N}}}\sum_{i=1}^n d(\gamma(t_{i-1}),\gamma(t_{i}))$$
\end{defn}

From the definition of the length of a path we can construct intrinsic metrics on $(\fpmultimodules,d_I)$.

\begin{defn}[Intrinsic Metric]
\label{defn:IntrinsicMetric}
Given a metric $d$ on the space of multiparameter modules we define an associated \textcolor{Maroon}{intrinsic metric} $\hat{d}$ with respect to a collection of admissable paths $\mathcal{C}$ (where $\mathcal{C}$ contains all constant paths and is closed under concatenations of paths).

\begin{equation}
    \hat{d}(M,N) = \inf_{\substack{\gamma \in \mathcal{C}\\ \gamma(0)=M \\ \gamma(1)=N}} \ell_d(\gamma)
\end{equation}
\end{defn}

Since the space of finitely presented multiparameter modules equipped with the interleaving distance is a path metric space (\Cref{cor:PathMetricSpace}) there exists a distance realising path between each pair of modules and so $\hat{d}_I = d_I$.
Thus the local result of \Cref{thm:LocalEquivalence} extends to a global result about intrinsic metrics.

\begin{thm}[Equivalence of Intrinsic Metrics]
Let $\hat{d}_0$ denote the intrinsic matching distance with respect to the collection of admissable paths $\mathcal{C} = \{\gamma : [0,1] \to (\fpmultimodules,d_I) : \gamma \text{ is } d_I\text{-continuous}\}$. The intrinsic metrics $\hat{d}_0$ and $d_I$ are globally bi-Lipschitz equivalent so that for all $M,N\in \fpmultimodules$:
$$  \frac{1}{34}d_I(M,N) \leq \hat{d}_0(M,N) \leq d_I(M,N)$$
\label{thm:GlobalEquivalence}
\end{thm}

\begin{proof}
The upper bound $\hat{d_0}\leq d_I$ follows from $d_0 \leq d_I$. Indeed for any path $\gamma$ we have that between modules $M$ and $N$

$$\ell_{d_0}(\gamma) = \sup_{\substack{0=t_0 < t_1 < ...< t_n = 1 \\ n\in \mathbb{N}}}\sum_{i=1}^n d_0(\gamma(t_{i-1}),\gamma(t_{i})) \leq \sup_{\substack{0=t_0 < t_1 < ...< t_n = 1 \\ n\in \mathbb{N}}}\sum_{i=1}^n d_I(\gamma(t_{i-1}),\gamma(t_{i})) = \ell_{d_I}(\gamma). $$
$\hat{d_0} = \inf_\gamma \ell_{d_0}(\gamma)$ and \Cref{cor:PathMetricSpace} implies that $d_I =  \inf_\gamma \ell_{d_I}(\gamma)$ thus $\hat{d_0}\leq d_I$.

For the lower bound we will show the length of any path between $M$ and $N$ is bounded below by $\frac{1}{34} d_I(M,N)$.
Let $\gamma \in \mathcal{C}$ with $\gamma(0)=M$ and $\gamma(1)=N$. For every $t\in [0,1]$ let $I_t \subset \gamma^{-1} B(\gamma(t),\frac{c_{\gamma(t)}}{4})$, be the largest interval containing $t$ in the preimage of the radius $\frac{c_{\gamma(t)}}{{4}}$ ball about the module $\gamma(t)$ in $\fpmultimodules$. By compactness of $[0,1]$ there is a minimal finite subcover $\{I_{t_1} , ...,I_{t_{2k+1}},..., I_{t_{2n+1}}\}$. Let $t_{2k} \in I_{t_{2k-1}}\cap I_{t_{2k+1}}$. Then by the local result \Cref{thm:LocalEquivalence} $d_0(\gamma(t_{2k}),\gamma(t_{2k\pm1})) \geq \frac{1}{34} d_I(\gamma(t_{2k}),\gamma(t_{2k\pm1}))$. Hence we observe that $ \ell(\gamma) \geq \frac{1}{34} d_I(M,N)$.
\end{proof}

\section{Application to Interlevel Set Persistence and Reeb Graphs}
\label{sec:InterlevelReeb}

In this section we outline an application of \Cref{thm:LocalEquivalence} and \Cref{thm:GlobalEquivalence} to interlevel set persistence modules \cite{carlsson_zigzag_2009, botnan_algebraic_2018, hutchison_robustness_2010} and Reeb Graphs. This section is intended to illustrate the generality of \Cref{thm:LocalEquivalence} and its applicability to data structures which arise in topological data analysis, even those which a priori do not have the structure of finitely presented multiparameter persistence modules.

Reeb graphs are a commonly studied object in data science and can be viewed as a particular instance of interlevel set persistence modules. Metrics on Reeb Graphs have been studied extensively \cite{carriere_local_2017, de_silva_categorified_2016, bauer_strong_2015, bauer_measuring_2014} and there is a trade off between discriminating power and computability of these metrics.
Theoretically sound but computably infeasible metrics exist for Reeb Graphs: such as the functional distortion distance, $d_\text{FD}$ and the Reeb Graph interleaving distance $d_I^\textsf{Reeb}$. It is shown in \cite{bauer_strong_2015} that $d_I^\textsf{Reeb}$ and $d_\text{FD}$ are bi-Lipschitz equivalent.
Computable, yet incomplete pseudo-metrics on Reeb Graphs include the bottleneck distance $d_B$ \cite{bauer_measuring_2014}. A comparison of the bottleneck distance and functional distortion distance is studied in \cite{carriere_local_2017}, and a global equivalence of intrinsic metrics is established.

We can associate a $2$-parameter persistence modules to an interlevel set persistence module. However, the 2-parameter persistence module associated to a non-trivial interlevel set persistence module is not finitely presented. Naively embedding interlevel set persistence modules into $\mathbb{R}^2$-indexed persistence modules gives rise to 2-parameter persistence modules with an uncountable number of generators. However, using the block decomposition of interlevel set persistence modules \cite{botnan_algebraic_2018, cochoy_decomposition_2020, bjerkevik_computing_2019} we may extend the 2-parameter persistence modules interlevel set persistence modules so that they are amenable to the arguments in \Cref{sec:LocalEquivalence} which apply to finitely presented multiparameter persistence modules.


\begin{defn}[Interlevel Set Persistence Module]
An \textcolor{Maroon}{interlevel set persistence module} is an element of the functor category $\vec{vect}^{\textbf{Int}(\mathbb{R})}$, where $\text{Int}(\mathbb{R})$ denotes the poset of non-empty open intervals in $\mathbb{R}$.
\end{defn}

\begin{ex}[Interlevel Set Persistence Modules]
Interlevel set persistence modules naturally arise from real valued functions on a topological space. Let $f:X\to \mathbb{R}$ and let $H$ denote a homology functor with coefficients in a field. The functor $(a,b)\mapsto H(f^{-1}((a,b)))$ is an interlevel set persistence module.
\label{ex:IntMod}
\end{ex}

One can define an interleaving distance for interlevel set persistence modules analogously to the interleaving distance for multiparameter modules \Cref{def:InterleavingDistance}. In this setting the translation endofunctors $T_\varepsilon: {\textbf{Int}(\mathbb{R})} \to {\textbf{Int}(\mathbb{R})}$ augment the intervals by $\varepsilon$: $T_\varepsilon((b,d))= (b-\varepsilon,d+\varepsilon)$. It is straight forward to show that interlevel set persistence modules isometrically embed as 2-parameter persistence modules.

\begin{lem}[Interlevel Set Persistence Module Embedding \cite{botnan_algebraic_2018}]
Let $\mathbb{U}$ denote the subposet $\{x_2 \geq -x_1 \}\subset \mathbb{R}^2$. The space of interlevel set persistence modules $\vec{vect}^{\textbf{Int}(\mathbb{R})}$ equipped with the interleaving distance embeds isometrically into $\vec{vect}^\mathbf{U}$, using the identification $\mathbf{U} \cong \mathbb{R}^\text{op} \times \mathbb{R}_\geq \cong \text{Int}(\mathbb{R})$.
\end{lem}

If we hope to associate a finitely presented module to an interlevel set persistence module we need to control the behaviour of interlevel set persistence modules. \Cref{defn:findet} provides the necessary criteria to control interlevel set persistence modules. 

\begin{defn}[Finitely Determined]
A module $ M \in \pfdmultimodules$ is \textcolor{Maroon}{finitely determined} if there exists a grid function 
$\mathcal{G}: [k_1] \times [k_2] \to \mathbb{R}^2$ such that for all 
$\vec{a}\leq \vec{b}$ with $\vec{a},\vec{b}\in (\mathcal{G}(i),\mathcal{G}(i+1))\times (\mathcal{G}(j),\mathcal{G}(j+1))=: \textsf{Cell}_{i,j}$ the internal morphism $M(\vec{a}\leq \vec{b})$ is an isomorphism, and for all $\vec{a} \not\geq \mathcal{G}((1,1)), M(\vec{a}) = 0$. For a subset $A\subset \mathbb{R}^2$ we shall say a module $M \in \vec{vect}^\vec{A}$ is finitely determined if there exists a grid function $\mathcal{G}: [k_1] \times [k_2] \to \mathbb{R}^2$ such that for all $\vec{a}\leq \vec{b}$ with $\vec{a},\vec{b}\in (\mathcal{G}(i),\mathcal{G}(i+1))\times (\mathcal{G}(j),\mathcal{G}(j+1)) \cap A$ the internal morphism $M(\vec{a}\leq \vec{b})$ is an isomorphism.
\label{defn:findet}
\end{defn}

Following \Cref{defn:findet} we will say that an interlevel set persistence modules $M \in \vec{vect}^{\textbf{Int}(\mathbb{R})}$ module is finitely determined if the associated module $M \in \vec{vect}^\mathbf{U}$ is finitely determined. We will denote the collection of finitely determined interlevel set persistence modules by  $\vec{vect}_\text{fin}^{\textbf{Int}(\mathbb{R})}$. We believe that this constructibility criterion is not too restrictive. For example, the interlevel set persistence module associated to a scalar function $f:X\to \mathbb{R}$ will be finitely determined if $X$ is a compact manifold and $f$ is Morse (\Cref{ex:IntMod}); alternatively if $X$ is a compact semi-algebraic subset of $\mathbb{R}^n$ and $f$ is a projection onto one of the coordinate values \cite{de_silva_categorified_2016}.

\begin{defn}[Block Decomposable \cite{botnan_algebraic_2018}]
Let us define the following notation for intervals of $\mathbf{U}$ which we shall call \textcolor{Maroon}{$\mathbf{U}$-blocks}:
\begin{align*}
    (a,b)_\text{BL} &= \{(x,y)\in \mathbf{U} : x\in (-b,-a), y \in (a,b)\}\\
    [a,b)_\text{BL} &= \{(x,y)\in \mathbf{U} : y \in [a,b)\}\\
    (a,b]_\text{BL} &= \{(x,y)\in \mathbf{U} : x \in [-b,a)\}\\
    [a,b]_\text{BL} &= \{(x,y)\in \mathbf{U} : x \in [-b,\infty), y\in [a,\infty)\}
\end{align*}

We say a module $M \in  \vec{vect}^{\mathbf{R}^2}$ is \textcolor{Maroon}{$\mathbf{U}$-block decomposable} is it is interval decomposable and each of the intervals in the decomposition is a $\mathbf{U}$-block.

\end{defn}

We want to work in the subclass of finitely presented persistence modules. To this end, we define the extension of $\mathbf{U}$-blocks, which takes each $\mathbf{U}$-block to a finitely presented interval of $\mathbb{R}^2$.

\begin{defn}[$\mathbb{U}$-Block Extension]
We define the following extensions of $\mathbf{U}$-blocks to intervals of $\mathbf{R}^2$:
\begin{align*}
    \overline{(a,b)}_\text{BL} &= \{(x,y)\in \mathbf{R}^2 : x \in [-b,-a), y \in [a,b)\}\\
    \overline{[a,b)}_\text{BL} &= \{(x,y)\in \mathbf{R}^2 : x \in [-b, \infty ), y \in [a,b)\}\\
    \overline{(a,b]}_\text{BL} &= \{(x,y)\in \mathbf{R}^2 : x \in [-b,a), y\in [a,\infty )\}\\
    \overline{[a,b]}_\text{BL} &= \{(x,y)\in \mathbf{R}^2 : x \in [-b,\infty ), y\in [a,\infty )\}
\end{align*}

\end{defn}

Using this extension we can assign a finitely presented module  $\overline{M} \in  \vec{vect}_\text{fin}^{\mathbf{R}^2}$ to every $\mathbf{U}$-block decomposable module $M \in  \vec{vect}^{\mathbf{R}^2}$. Explicitly if $M \cong \bigoplus_j \mathds{1}^{I_j}$ (with $I_j$ all $\mathbf{U}$-blocks) we take $\overline{M}\cong \bigoplus_j \mathds{1}^{\overline{I_j}}$.

The interleaving distance between $\mathbf{U}$-block decomposable modules can be computed by matching the blocks in their decomposition \cite{bjerkevik_stability_2016}. More precisely, given $\mathbf{U}$-block decomposable modules $M \cong \bigoplus_j \mathds{1}^{I_j}$ and $N\cong \bigoplus_k \mathds{1}^{J_k}$ if $d_I(M,N)= \varepsilon$ then there is a matching between the blocks $\{I_j\}$ and $\{J_k\}$ such that if $I_j$ is matched to $J_k$ then $I_j$ and $J_k$ are of the same type, $d_I(\mathds{1}^{I_j},\mathds{1}^{J_k}) \leq \varepsilon$ and any unmatched block is $\varepsilon$-interleaved with the zero module \cite{bjerkevik_stability_2016}.

It is straight forward to check that if $I_j$ and $J_k$ are the same type then: $$d_I(\mathds{1}^{I_j},\mathds{1}^{J_k}) \leq d_I(\mathds{1}^{\overline{I_j}},\mathds{1}^{\overline{J_k}}) \leq 2d_I(\mathds{1}^{I_j},\mathds{1}^{J_k})$$

and so for  $\mathbf{U}$-block decomposable modules $M,N$:
$$d_I(M,N) \leq d_I(\overline{M},\overline{N}) \leq 2d_I(M,N)$$

We can derive an interlevel set persistence module from a Reeb Graph. In \Cref{fig:ReebEmbedding} we sketch the 2-parameter $\mathbf{U}$-block decomposable module associated to a simple Reeb Graph together with its extension to a finitely presented 2-parameter module.
Since we can extend interlevel set persistence modules to finitely presented 2-parameter persistence modules we yield the following result as an application of \Cref{thm:LocalEquivalence}:

\begin{cor}[Local Equivalence of Interlevel Set Persistence Module Metrics]
Suppose $M,N \in \vec{vect}_\text{fin}^{\textbf{Int}(\mathbb{R})}$ are such that $d_I(M,N)=\varepsilon$ and $\frac{c_{\overline{M}}}{4(34\kappa +1)} \geq  \varepsilon$ then  for all $\kappa \in [0, \frac{1}{34})$ the matching distance between their extended modules is bounded below ${d}_0(\overline{M},\overline{N}) \geq \kapeps$.
\label{cor:LocalEquivalenceInterlevel}
\end{cor}

\begin{proof}
Apply \Cref{thm:LocalEquivalence} to the modules $\overline{M}$ and $\overline{N}$ using the fact that $d_I(\overline{M},\overline{N}) \leq 2d_I(M,N)$.
\end{proof}

\Cref{cor:LocalEquivalenceInterlevel} provides an example of a local equivalence result which may be derived from our general result \Cref{thm:LocalEquivalence}. We relied on the fact that the poset $\textbf{Int}(\mathbb{R})$ naturally includes into $\mathbf{R}^2$ and that sensible constructibility criteria for interlevel set persistence modules give rise to finitely presented 2-parameter persistence modules. In this particular instance, interlevel set persistence modules and Reeb Graphs have been well studied and have additional structure which has been exploited to yield sharper results than \Cref{cor:LocalEquivalenceInterlevel} \cite{bjerkevik_stability_2016,hutchison_robustness_2010,botnan_algebraic_2018,carriere_local_2017,bauer_strong_2015}. We would be interested to explore other posets $\mathbf{P}$ which naturally include into $\mathbf{R}^n$ for some $n$, and the corresponding constructibility criteria of modules $\vec{vect}^\vec{P}$ for which this inclusion would give rise to finitely presented modules. This would take advantage of the generality of our main result \Cref{thm:LocalEquivalence} and yield results for more complicated topological data structures than Reeb Graphs. A starting point would be to explore Reeb Spaces.

\begin{figure}
\centering

\includegraphics[width=0.9\linewidth]{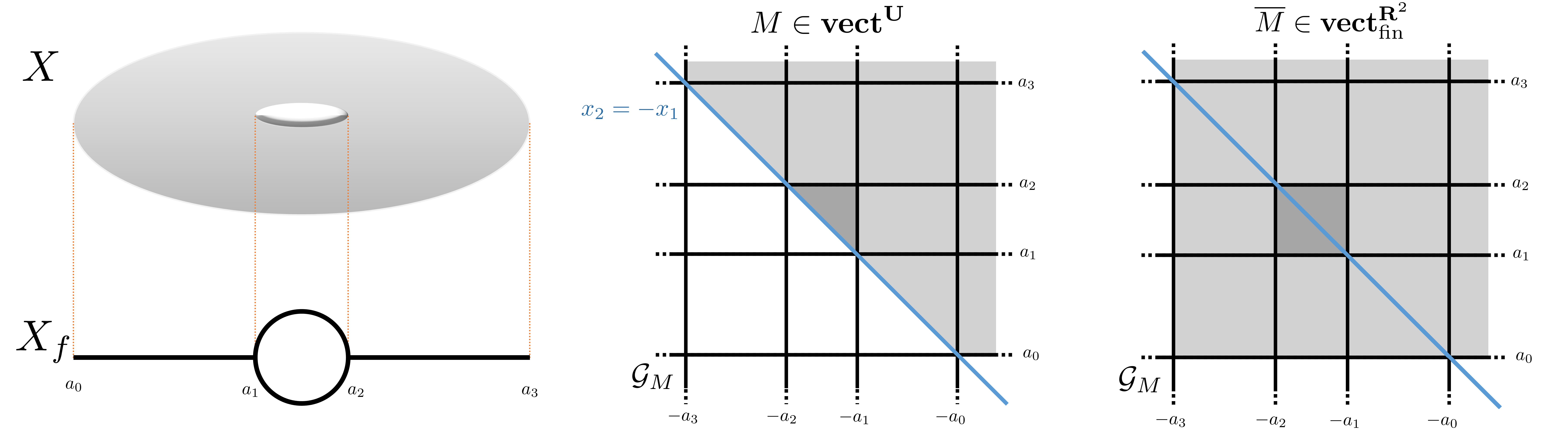}

\caption{An illustration of the Reeb graph $X_f$, where $X = \mathbb{T}^2$ is the 2-torus equipped with the projection $f$ to the real line. The Reeb graph $X_f$ has an associated $\vec{vect}$ valued-cosheaf on the real line which isometrically embedded into $\vec{vect}^{\mathbf{U}}$ (both spaces equipped with the interleaving distance). Using the block decomposition of an interlevel set persistence module $M\in \vec{vect}^{\mathbf{U}}$ we may extend $M$ to $\overline{M}\in \fpbimodules$.}
\label{fig:ReebEmbedding}
\end{figure}

\printbibliography

\end{document}